%% file: jerome_nn_optimal_control.tex
\documentclass[10pt,reqno,final]{amsart}
\usepackage{amsfonts,latexsym,url,amsmath,amssymb,comment,enumitem}
\usepackage{upgreek}
\excludecomment{commentforus}
\usepackage{subfigure}
\usepackage{url,centernot}
\usepackage{graphicx,color}
\usepackage[color,notref,notcite]{showkeys}
\usepackage{bm}
\usepackage{stmaryrd}
\usepackage{citesort}

\DeclareMathAlphabet\mathbfcal{OMS}{cmsy}{b}{n}

\newcommand{\mathbi}[1]{{\boldsymbol #1}}

\def\R{\mathbb{R}}

\def\err{\mathsf{err}}

\def\P{\mathbfcal{P}}
\def\Proj{\P_{\!\mesh}}
\def\<{\langle}
\def\>{\rangle}

\def\dsp{\displaystyle} 
\def\div{{\rm div}}

\def\trrec{\mathbb{T}}

\def\norm#1#2{\Vert#1\Vert_{#2}}

\newcounter{cst}
\def \ctel#1{C_{\refstepcounter{cst}\label{#1}\thecst}}
\def \cter#1{C_{\ref{#1}}}

\def\bt{\begin{theorem}}
\def\et{\end{theorem}}
\def\bl{\begin{lemma}}
\def\el{\end{lemma}}
\def\bc{\begin{corollary}}
\def\ec{\end{corollary}}
\def\bd{\begin{definition}}
\def\ed{\end{definition}}
\def\br{\begin{remark}}
\def\er{\end{remark}}
\def\tW{\widetilde{W}}

\def\ad{{\rm ad}}
\def\Prh{{\rm Pr}_h}
\def\Uad{\mathcal U_\ad}
\def\Uadh{\mathcal U_{\ad,h}}
\def\Uh{\mathcal U_{h}}
\def\bU{\overline{U}}

\def\bP{\overline{P}}

\def\cv{K}
\newcommand{\polyd}{{\mathfrak{T}}}

\newcommand{\mesh}{{\mathcal T}}

\newcommand{\edge}{{\sigma}}

\newcommand{\edges}{{\mathcal E}}              
\newcommand{\edgescv}{{{\edges}_\cv}}

\newcommand{\centers}{\mathcal{P}}
\newcommand{\x}{\mathbi{x}}

\newcommand{\bu}{\overline{u}}
\def\bub{\overline{u}_b}
\def\bubh{\overline{u}_{b,h}}
\newcommand{\bvarphi}{\mathbi{\varphi}}

\newcommand{\be}{\begin{equation}}
\newcommand{\ee}{\end{equation}}
\newcommand{\disc}{{\mathcal D}}

\renewcommand{\O}{\Omega}

\def\dr{\partial}
\def\bfn{\mathbf{n}}

\renewcommand{\d}{{\rm d}}

\newcommand{\ba}{\begin{array}{llll}   }
\newcommand{\bac}{\begin{array}{c}}
\newcommand{\bari}{\begin{array}{r}}
\newcommand{\ea}{\end{array}}

\def\SCAL#1#2{\left\llbracket\,#1\,\left|\,#2\,\right.\right\rrbracket}
\newcommand{\NORM}[1]{{\left\vert\kern-0.25ex\left\vert\kern-0.25ex\left\vert #1 
    \right\vert\kern-0.25ex\right\vert\kern-0.25ex\right\vert}}

\def\minmod{\mathop{\rm minmod}}

\newcommand{\by}{\overline y}
\newcommand{\bp}{\overline  p}
\newcommand{\tu}{\widetilde{u}}
\newcommand{\ud}{\overline{u}_d}
\newcommand{\yd}{\overline{y}_d}
\newcommand{\udh}{\widehat{u}_d}

\def\hat#1{\widehat{#1}}

\def\assum#1{{\rm\textbf{(A#1)}}}

\newtheorem{theorem}{Theorem}[section]
\newtheorem{remark}[theorem]{Remark}

\newtheorem{lemma}[theorem]{Lemma} 
\newtheorem{definition}[theorem]{Definition}
\newtheorem{proposition}[theorem]{Proposition}
\newtheorem{corollary}[theorem]{Corollary}

\numberwithin{equation}{section}
\def\WS{{\rm WS}}

\def\Xint#1{\mathchoice
{\XXint\displaystyle\textstyle{#1}}%
{\XXint\textstyle\scriptstyle{#1}}%
{\XXint\scriptstyle\scriptscriptstyle{#1}}%
{\XXint\scriptscriptstyle\scriptscriptstyle{#1}}%
\!\int}
\def\XXint#1#2#3{{\setbox0=\hbox{$#1{#2#3}{\int}$ }
\vcenter{\hbox{$#2#3$ }}\kern-.6\wd0}}

\def\dashint{\Xint-}

\newcounter{cexp}
\catcode`\@=11
\def\terml#1{T_{\refstepcounter{cexp}\@bsphack
\protected@write\@auxout{}%
           {\string\newlabel{#1}{{\thecexp}{\thepage}}}\thecexp}}
\catcode`\@=12

\begin{document}
\title[GDM for optimal control problems, and super-convergence]{The gradient discretisation method for optimal control problems, with super-convergence
for non-conforming finite elements and mixed-hybrid mimetic finite differences}
\author{J\'erome Droniou}
\address{School of Mathematical Sciences, Monash University, Clayton,
Victoria 3800, Australia.
\texttt{jerome.droniou@monash.edu}}
\author{Neela Nataraj}
\address{Department of Mathematics, Indian Institute of Technology Bombay, Powai, Mumbai 400076, India.
\texttt{neela@math.iitb.ac.in}}
\author{Devika Shylaja}
\address{IITB-Monash Research Academy, Indian Institute of Technology Bombay, Powai, Mumbai 400076, India.
\texttt{devikas@math.iitb.ac.in}}

\date{\today}


\maketitle

\begin{abstract}
In this paper, optimal control problems governed by diffusion equations with Dirichlet and Neumann boundary conditions are investigated in the framework of the gradient discretisation method. Gradient schemes are defined for the optimality system of the control problem. Error estimates for state, adjoint and control variables are derived. Superconvergence results for gradient schemes under realistic regularity assumptions on the exact solution is discussed. These super-convergence results are shown to apply to non-conforming $\mathbb{P}_1$ finite elements,
and to the mixed/hybrid mimetic finite differences. Results of numerical experiments are demonstrated for the conforming, non-conforming and mixed-hybrid mimetic finite difference schemes.
\end{abstract}

\medskip

{\scriptsize

\smallskip

}

\section{Introduction}

This paper is concerned with numerical schemes for second order distributed optimal control problems governed by diffusion equations with Dirichlet and Neumann boundary conditions (BC). We present
basic convergence results and super-convergence results for the state, adjoint and control variables.
Our results cover various numerical methods, including conforming Galerkin methods, non-confor\-ming
finite elements, and mimetic finite differences. This is achieved by using the framework of the
gradient discretisation method.

\medskip

The gradient discretisation method (GDM) is a generic framework for the convergence analysis of numerical methods for diffusion equations.
The GDM consists in replacing the continuous space and
operators by discrete ones in the weak formulation of the partial differential equation (PDE). The set of discrete elements thus chosen is called a gradient discretisation (GD),
and the scheme obtained by using these elements is a gradient scheme (GS). The variety of
possible choices of GDs results in as many different GSs, allowing the GDM to cover a wide
range of numerical methods (finite elements, mixed finite elements, finite volume, mimetic methods, etc.).
Reference
\cite{DEH15} presents the methods known to be GDMs, and \cite{dro-14-deg,dro-14-sto,dro-12-gra,DL14,eym-12-stef,eym-12-sma,zamm2013} provide a few models on which the convergence analysis can be carried out within this framework; see also \cite{koala} for a complete presentation of the GDM for various boundary conditions and models.

In this paper, we discretise optimal control problems governed by diffusion equations using the gradient
discretisation method. A gradient scheme is defined for the optimal control problem by discretising the
corresponding optimality system, which involve state, adjoint and control variables.
Error estimates of two kinds are derived for the state, adjoint and control variables.
Firstly, we establish basic error estimates in a very generic setting. Secondly, considering
slightly more restrictive assumptions on the admissible control set, we derive super-convergence results
for all three variables. The theoretical results are substantiated by numerical experiments.

\medskip

Numerical methods for second-order optimal control problems governed by Dirichlet BC have been
studied in various articles (see, e.g., \cite{TAARGW,YCDY,AKBV,MTV,CMAR} for distributed control, \cite{TAMMJPAR,SMRRBV} for boundary control, and references therein). For conforming and mixed finite element methods, superconvergence result of control has been derived in \cite{YCheng,supercv_mixed,CMAR}. For the finite element analysis of Neumann boundary control problems, see \cite{TAJPAR,TAJPAR15,ECMMFT,KP15,MMAR}.
However, to the best of our knowledge, super-convergence has not been studied for classical non-conforming methods for control problems governed by second order linear elliptic problems. One of the consequences of
our generic analysis is to establish superconvergence results for several conforming and non-conforming
numerical methods covered by gradient schemes -- in particular, the classical 
Crouzeix-Raviart finite element method and the mixed-hybrid mimetic finite difference schemes.

Note that we only consider linear control problems here, but the GDM has been designed to also
deal with non-linear models. So for non-linear state equations that are amenable to error estimates
(e.g. the $p$-Laplace equation \cite[Theorem 3.28]{koala}), an adaptation of the results presented here
is conceivable.

\medskip

The paper is organised as follows. Subsection \ref{sec:model} defines the distributed optimal control problem governed by the diffusion equation with homogeneous Dirichlet BC. Two particular cases of our main results are stated in Subsection \ref{sec:schemes}; these cases cover non-conforming finite element methods and mixed-hybrid mimetic finite
difference schemes. In Section \ref{sec:defGS}, the GDM is introduced, the concept of GD is defined and the properties on the spaces and mappings that are important for the convergence analysis of the
resulting GS are stated. The basic error estimates and superconvergence results for the GDM applied to the control problems are presented in Section \ref{section3}.
The superconvergence for the control variable is obtained under a superconvergence assumption on the underlying state and adjoint equations which, if not already known, can be checked for
various gradient schemes by using the improved $L^2$ estimate of \cite{jd_nn}. Note that our superconvergence results are twofold. Under generic assumptions on the GD,
which allow for local mesh refinements, we prove an $\mathcal O(h^{2-\epsilon})$ super-convergence
with $\epsilon>0$ in dimension 2 and $\epsilon=1/6$ in dimension 3. Under an $L^\infty$-bound assumption on the
solution to the GS, which for most methods requires the quasi-uniformity of the meshes,
we prove a full $\mathcal O(h^2)$ convergence result.
Section \ref{sec:proofs} deals with the proof of the main results which are stated in Section \ref{section3}.
The superconvergence is established by following the ideas developed in \cite{CMAR} for conforming
finite elements. In Section \ref{section5}, the distributed and boundary optimal control problems with Neumann BC is presented. The GDM for Neumann BC is discussed in this section and the core properties that the GDs must satisfy to provide a proper approximation of given problem are highlighted. We present the results of some numerical experiments in Section \ref{sec_examples}. An appendix, Section \ref{sec:LinftyHMM} recalls the GD corresponding to mixed hybrid mimetic schemes and derives $L^{\infty}$ estimates. 

Since the paper deals with a variety of methods, we conclude this introduction with a list of some abbreviations used throughout the paper (see Table \ref{table:1}).
\begin{table}[h!]
\begin{center}
\begin{tabular}{|p{2cm}|p{10cm}|} 
		\hline
		Abbreviations & Meaning  \\ \hline
   	FE& Finite Element\\ \hline
	 GDM & Gradient Discretisation Method, see Section \ref{sec:defGS}\\ \hline
	 GD & Gradient Discretisation, see Definition \ref{def:GD}\\ \hline
	 GS & Gradient Scheme, see \eqref{base.GS} for PDEs\\ \hline
	 MFD & Mimetic Finite Difference method, a family of numerical schemes
	for diffusion equations on generic meshes. hMFD is the hybrid Mimetic Finite Difference method, a member of this family based on unknowns located in the cell and on the faces.\\ \hline
	 HMM & Hybrid Mimetic Mixed method, a family of numerical methods for diffusion
	equations on generic meshes, that contains hMFD as a particular case. See Section \ref{sec:LinftyHMM} for a presentation of the gradient discretisation corresponding to the HMM scheme.\\ \hline
	\end{tabular}
\end{center}
\caption{List of abbreviations}
\label{table:1}
\end{table}

\subsection{The optimal control problem for homogeneous Dirichlet BC} \label{sec:model}

Consider the distributed optimal control problem governed by the diffusion equation defined by
\begin{subequations}\label{cost_main}
\begin{align}
&  { \min_{u \in \Uad} J(y, u) \quad \textrm{ subject to } }  \label{cost}\\
&  {-\div(A\nabla y) = f+ u  \quad \mbox{ in } \Omega,  }\label{state1} \\
&  {\hspace{1.9cm}y = 0\quad\mbox{ on $\partial\O$},} \label{state2}
\end{align}
\end{subequations} 
where  $\Omega \subsetneq \R^n \; (n \ge 2)$ is a bounded domain with boundary $\partial \Omega$;
$y$ is the state variable, and $u$ is the control variable;
\be \label{costfunctional}
J(y, u) :=\frac{1}{2}\norm{y- \yd}{L^2(\O)}^2 + \frac{\alpha}{2} \norm{u-\ud}{L^2(\O)}^2 \ee
is the cost functional, $\alpha>0$ is a fixed regularization parameter, $\yd\in L^2(\O)$ is the desired state variable and  $\ud\in L^2(\O)$ is the desired control variable; $A:\O\rightarrow M_n(\R)$
is a measurable, bounded and uniformly elliptic matrix-valued function such
that (for simplicity purposes) $A(\x)$ is symmetric for a.e.\ $\x\in \O$; $f\in L^2(\O)$; $\Uad \subset L^2(\Omega)
$ is the non-empty, convex and closed set of admissible controls.

\medskip
It is well known that given $u \in \Uad$, there exists a unique weak solution $y(u) \in H^1_0(\Omega):=\{w\in H^1(\O)\,:w=0\mbox{ on }\partial\O\}$ of \eqref{state1}-\eqref{state2}.
That is, for $u \in \Uad$, there exists a unique  $y(u) \in H^1_0(\Omega)$ such that for all ${w}\in H^1_0(\O)$,
\begin{equation}\label{weak_state}
a(y(u),w)=\int_\O (f+u) {w} \d\x,
\end{equation}
where $a(z,w)=\int_\O A\nabla z\cdot\nabla w \d\x$. The term $y(u)$ is the {state} {associated} with the control $u$.

\medskip

In the following, we denote by $\|\cdot\|$ and $(\cdot,\cdot)$, the norm and
scalar product in $L^2(\O)$ (or $L^2(\O)^n$ for vector-valued functions).
The convex control problem \eqref{cost_main} has a unique weak solution $(\by,\bu) \in H^1_0(\O) \times \Uad $. Also there exists a co-state $\bp \in H^1_0(\O)$ such that the triplet $(\by, \bp, \bu) \in H^1_0(\O) \times H^1_0(\O) \times \Uad$ satisfies the Karush-Kuhn-Tucker (KKT) optimality conditions \cite[Theorem 1.4]{jl}:
\begin{subequations} \label{continuous}
\begin{align}
& a(\by,w) = (f + \bu, w) \,  & \forall \: w \in H^1_0(\O), \label{state_cont}\\
& a(w,\bp) = (\by-\yd, w)\,  &  \forall \: w \in H^1_0(\O), \label{adj_cont}\\
&(\bp+\alpha(\bu-\ud),v-\bu)\geq 0\,  & \forall \: v\in \Uad. \label{opt_cont}
\end{align}
\end{subequations}

\subsection{Two particular cases of our main results} \label{sec:schemes}

Our analysis of numerical methods for \eqref{continuous} is based
on the abstract framework of the gradient discretisation method. To give an idea of the extent
of our main results, let us consider two particular schemes, based
on a mesh $\mesh$ of $\O$. We assume here that $\Uad=\{v \in L^2(\O)\,:\,a\le v\le b \mbox{ a.e.}\}$ for some constants $a,b$ (possibly infinite)
and, to simplify the presentation, that $\ud=0$.
Set $P_{[a,b]}(s)=\min(b,\max(a,s))$. Let $\tu$ be a post-processed control, whose scheme-dependent definition is given below.

\begin{itemize}
\item nc$\mathbb{P}_1$/$\mathbb{P}_0$: $\mesh$ is a conforming
triangular/tetrahedral mesh, the state and adjoint unknowns $(\by,\bp)$
are approximated using non-conforming $\mathbb{P}_1$ finite elements, and
the control $\bu$ is approximated using piecewise constant functions on $\mesh$. 

\noindent We then let the post-processed continuous control be $\tu=\bu$.
\smallskip

\item hMFD \cite{ABV13}: $\mesh$ is a polygonal/polyhedral mesh,
the state and adjoint unknowns $(\by,\bp)$ are approximated using mixed-hybrid mimetic
finite differences (hMFD), and
the control $\bu$ is approximated using piecewise constant functions on $\mesh$.
The hMFD schemes form a sub-class of the hybrid mimetic mixed (HMM) methods \cite{dro-10-uni,dro-12-gra}
presented in the appendix (Section \ref{sec:LinftyHMM}); see also Section \ref{sec_examples} for
corresponding numerical tests.

\noindent We then define a post-processed continuous control
$\tu$ by 
\[
\tu_{|K}=P_{[a,b]}(-\alpha^{-1}\bp(\overline{\x}_K))\quad\mbox{ for all $K\in\mesh$},
\]
where $\overline{\x}_K$ denotes the centroid of the cell $K$.
\end{itemize}

In either case, the post-processed discrete control is $\tu_h=P_{[a,b]}(-\alpha^{-1}\bp_h)$, where $\bp_h$ denotes the discrete co-state. 
One of the consequences of our first main theorem is the following super-convergence
result on the control, under
standard regularity assumptions on the mesh and the data:
there exists $C$ that depends only on $\O$, $A$, $\alpha$, $a$, $b$, $\bu$,
and the shape regularity of $\mesh$, such that
\be\label{rate-cv.nc.hmm}
\norm{\tu-\tu_h}{}\le C h^r (1+\norm{\yd}{H^1(\O)}
+\norm{f}{H^1(\O)}+\norm{\ud}{H^2(\O)}),
\ee
where $r=2-\epsilon$ (for any $\epsilon>0$) if $n=2$, and $r=\frac{11}{6}$ if $n=3$.
This estimate is also valid for conforming $\mathbb{P}_1$ finite elements.

\smallskip

Under the additional assumption of quasi-uniformity of the mesh (that is, each cell has a measure
comparable to $h^n$), our second main theorem shows that
\eqref{rate-cv.nc.hmm} can be improved into a full quadratic rate of convergence:
\be\label{rate-cv.nc.hmm.h2}
\norm{\tu-\tu_h}{}\le C h^2 (1+\norm{\yd}{H^1(\O)}
+\norm{f}{H^1(\O)}+\norm{\ud}{H^2(\O)}).
\ee
The quasi-uniformity assumption prevents us from considering
local mesh refinement, so \eqref{rate-cv.nc.hmm.h2} is not ensured in these cases.
On the contrary, the rate \eqref{rate-cv.nc.hmm} still holds true for
locally refined meshes. Moreover, in dimension $n=2$, the $h^{2-\epsilon}$ rate
in \eqref{rate-cv.nc.hmm} is numerically indistinguishable from a full super-convergence $h^2$ rate.
If $n=3$, the $h^{\frac{11}{6}}$ rate of convergence
remains very close to $h^2$. To compare, for $h=10^{-6}$ (which is
well below the usual mesh sizes in 3D computational tests) we have $h^2/h^{\frac{11}{6}}=10^{-1}$.

Precise statement of the assumptions and the proofs of \eqref{rate-cv.nc.hmm} and \eqref{rate-cv.nc.hmm.h2} 
are given in Corollary \ref{cor:nc.hmm}.

\section{The gradient discretisation method for the control problem}\label{sec:defGS}

The gradient discretisation method (GDM) consists in writing numerical schemes,
called gradient schemes (GS), by replacing in the weak formulation
of the problem the continuous space and operators by discrete ones \cite{koala,dro-12-gra,eym-12-sma}. These discrete
space and operators are given by a gradient discretisation (GD).

\begin{definition}[Gradient discretisation for homogeneous Dirichlet BC]\label{def:GD}
A gradient discretisation for homogeneous Dirichlet BC is given by $\disc=(X_{\disc,0},\Pi_\disc,\nabla_\disc)$ such that
\begin{itemize}
\item the set of discrete unknowns (degrees of freedom) $X_{\disc,0}$ is a finite dimensional real vector space,
\item $\Pi_\disc:X_{\disc,0}  \rightarrow L^2(\O)$ is a linear mapping that reconstructs a function from
the degrees of freedom,
\item $\nabla_\disc:X_{\disc,0}  \rightarrow L^2(\O)^n$ is a linear mapping that reconstructs
a gradient from the degrees of freedom. It must be chosen such that
$\norm{\nabla_\disc \cdot}{}$ is a norm on $X_{\disc,0}$.
\end{itemize}
\end{definition}
Let $\disc=(X_{\disc,0},\Pi_\disc,\nabla_\disc)$ be a GD in the sense of the above definition. If $F \in L^2(\O) $, then the related gradient scheme for a linear elliptic problem
\be\label{base}
\left\{
\ba
-\div(A\nabla \psi)=F&\mbox{ in $\O$},\\
\psi=0&\mbox{ on $\dr\O$}
\ea
\right.
\ee
is obtained by writing the
weak formulation of \eqref{base} with the continuous spaces, function and gradient
replaced with their discrete counterparts:
\be\label{base.GS}
\mbox{Find $\psi_\disc\in X_{\disc,0}$ such that, for all $w_\disc\in X_{\disc,0}$, }
a_\disc(\psi_\disc,w_\disc)=(F,\Pi_\disc w_\disc),
\ee
where $a_\disc(\psi_\disc,w_\disc)=\int_\O A\nabla_\disc \psi_\disc\cdot\nabla_\disc w_\disc\d\x$.

\begin{remark}
The inclusion of a scheme $\mathcal S$ for \eqref{base} into the GDM
consists in finding a suitable gradient discretisation $\disc$ such that $\mathcal S$
is algebraically identical to \eqref{base.GS}. We refer to \cite{DEH15} for the proofs that various
classical methods fit into the GDM.
\end{remark}

Let $\Uh$ be a finite-dimensional subspace of $L^2(\Omega)$,
and $\Uadh = \Uad \cap \Uh$. A gradient discretisation $\disc$ being given,
the corresponding GS for \eqref{continuous} consists in 
seeking $(\by_{\disc}, \bp_{\disc}, \bu_{h})\in X_{\disc,0} \times X_{\disc,0} \times \Uadh$ such that
\begin{subequations} \label{discrete_kkt}
\begin{align}
& a_{\disc}(\by_{\disc},w_{\disc}) = (f + \bu_{h}, \Pi_\disc w_{\disc}) \,
&\forall \: w_{\disc} \in  X_{\disc,0},  \label{discrete_state} \\
& a_{\disc}(w_{\disc},\bp_{\disc}) = (\Pi_\disc \by_{\disc}-\yd, \Pi_\disc w_{\disc})  \,&  \forall \: w_{\disc} \in  X_{\disc,0}, \label{discrete_adjoint} \\
&(\Pi_\disc\bp_{\disc} +\alpha(\bu_h-\ud),v_h-\bu_h)\geq     0\,
&\forall \: v_h \in  \Uadh. \label{opt_discrete}
\end{align}
\end{subequations} 

Arguing as in \cite[Theorem 2.25]{tf}, it is straightforward to see that \eqref{discrete_kkt} is equivalent to the
following minimisation problem:
\begin{equation}\label{min_pb_discrete}
\begin{aligned}
&  { \min_{u_h \in \Uadh} \frac{1}{2}\norm{\Pi_\disc y_\disc -\yd}{}^2 + \frac{\alpha}{2}\norm{u_h-\ud}{}^2
\quad \textrm{ subject to } } \\
&  y_\disc\in X_{\disc,0} \mbox{ and, for all $w_\disc\in X_{\disc,0}$, }
 a_\disc(y_\disc,w_\disc)=(f+u_h,\Pi_\disc w_\disc).
\end{aligned}
\end{equation}
Existence and uniqueness of a solution to \eqref{min_pb_discrete}, and thus to \eqref{discrete_kkt},
follows from standard variational theorems.

\subsection{Results on the GDM for elliptic PDEs}

We recall here basic notions and known results on the GDM for elliptic PDEs.

The accuracy of a GS \eqref{base.GS} is measured by three quantities.
The first one, which ensures the \emph{coercivity} of the method, controls
the norm of $\Pi_\disc$.
\be\label{def.CD}
C_\disc := \max_{w\in X_{\disc,0}\setminus \{0\}} \frac{\norm{\Pi_\disc w}{}}{\norm{\nabla_\disc w}{}}.
\ee
The second measure involves an estimate of the interpolation error, called the GD-\emph{consistency} 
(or consistency, for short) in the framework of the GDM. It corresponds to the interpolation error in the finite elements
nomenclature.
\be\label{def.SD}
\begin{aligned}
&\forall \varphi\in H^1_0(\O),\,\,
S_\disc(\varphi)=\min_{w\in X_{\disc,0}}\left(\norm{\Pi_\disc w-\varphi}{}
+\norm{\nabla_\disc w-\nabla\varphi}{}\right).
\end{aligned}
\ee

Finally, we measure the \emph{limit-conformity} of a GD by
defining
\be\label{def.WD}
\begin{aligned}
&\forall \bvarphi\in H_{\div}(\O),\,\,\,
W_\disc(\bvarphi)=\max_{w\in X_{\disc,0}\setminus \{0\}}\frac{1}{\norm{\nabla_\disc w}{}}
\left| \tW_\disc(\bvarphi, w)\right|,
\end{aligned}
\ee
where $H_{\div}(\O)=\{\bvarphi\in L^2(\O)^n\,:\,\div(\bvarphi)\in  L^2(\O)\}$ and
\be\label{def:tW}
\tW_\disc(\bvarphi, w) =\dsp\int_\O \left( \Pi_\disc w \: \div(\bvarphi)
+\nabla_\disc w\cdot\bvarphi\right) \d\x.
\ee
We use the following notation.
\begin{equation}\label{notation:lesssim}
\begin{aligned}
&X\lesssim Y\mbox{ means that }X\le CY\mbox{ for some $C$ depending}\\
&\mbox{only on $\O$, $A$ and an upper bound of $C_\disc$}.
\end{aligned}
\end{equation}

The following basic error estimate on GSs is standard, see \cite[Theorem 3.2]{koala}.

\begin{theorem}\label{th:error.est.PDE} Let $\disc$ be a GD in the
sense of Definition \ref{def:GD}, $\psi$ be the solution
to \eqref{base}, and $\psi_\disc$  be the solution to \eqref{base.GS}. Then
\begin{align}
\norm{\Pi_\disc \psi_\disc-\psi}{}+\norm{\nabla_\disc \psi_\disc-\nabla\psi}{}
\lesssim \WS_\disc(\psi),
\end{align}
where 
\be \label{def.ws}
\WS_\disc(\psi)= W_\disc(A\nabla \psi)+S_\disc(\psi)
\ee
$S_\disc$ is defined by \eqref{def.SD} and $W_\disc$ is defined by \eqref{def.WD}.
\end{theorem}

\begin{remark}[Rates of convergence for the PDE]\label{rates.PDE} For all classical first-order methods based on meshes with mesh parameter ``$h$'',
$\mathcal O(h)$ estimates can be obtained for $W_\disc(A\nabla \psi)$ and $S_\disc(\psi)$, 
if $A$ is Lipschitz continuous and $\psi\in H^2(\O)$ (see \cite[Chapter 8]{koala}).
Theorem \ref{th:error.est.PDE} then gives a linear rate of convergence
for these methods.
\end{remark}

\section{Main results: basic error estimate and super-convergence} \label{section3}
In this section, the main contributions are stated and the assumptions are discussed in details. The proofs of the results are presented in Section 4.
\subsection{Basic error estimate for the GDM for the control problem}

To state the error estimates, let $\Prh:L^2(\O) \rightarrow \Uh$ be the $L^2$ orthogonal projector on $\Uh$ for the standard scalar product.

\begin{theorem}[Control estimate]\label{theorem.control.DB}
Let $\disc$ be a \textrm{GD}, 
$(\by,\bp,\bu)$ be the solution to \eqref{continuous} and
$(\by_\disc,\bp_\disc,\bu_h)$ be the solution to \eqref{discrete_kkt}.
We assume that
\be\label{stab.proj}
\Prh(\Uad)\subset \Uadh.
\ee
Then,
\be\label{est.basic.u}
\begin{aligned}
\sqrt{\alpha}\norm{\bu-\bu_h}{}\lesssim{}& \sqrt{\alpha}\norm{\alpha^{-1}\bp-\Prh(\alpha^{-1}\bp)}{}+(\sqrt{\alpha}+1)\norm{\bu-\Prh\bu}{}\\
&+\sqrt{\alpha}\norm{\ud-\Prh\ud}{}+\frac{1}{\sqrt{\alpha}} \WS_\disc(\bp)+\WS_\disc(\by).
\end{aligned}
\ee
\end{theorem}

\begin{proposition}[State and adjoint error estimates]\label{theorem.state.adj.DB}
Let $\disc$ be a GD, 
$(\by,\bp,\bu)$ be the solution to \eqref{continuous} and
$(\by_\disc,\bp_\disc,\bu_h)$ be the solution to \eqref{discrete_kkt}.
Then the following error estimates hold:
\begin{align}
\label{est.basic.y}
\norm{\Pi_\disc \by_\disc-\by}{}+\norm{\nabla_\disc \by_\disc-\nabla\by}{}   \lesssim{}& \norm{\bu-\bu_h}{}+ \WS_\disc(\by),\\
\label{est.basic.p}
\norm{\Pi_\disc \bp_\disc-\bp}{}+\norm{\nabla_\disc \bp_\disc-\nabla\bp}{} \lesssim{}& \norm{\bu-\bu_h}{} + \WS_\disc(\by) + \WS_\disc(\bp).
\end{align}		
\end{proposition}

\begin{remark}[Rates of convergence for the control problem]\label{rates.control}
Owing to Remark  \ref{rates.PDE}, under sufficient smoothness assumption on $\ud$, if $A$ is Lipschitz continuous and $(\by,\bp,\bu)\in H^2(\O)^2 \times H^1(\O)$
then \eqref{est.basic.u}, \eqref{est.basic.y} and \eqref{est.basic.p} give linear rates of convergence for all classical first-order methods.
\end{remark}

\subsection{Super-convergence for post-processed controls} \label{sec:superconvergence}

We consi\-der here the case $n\le 3$, and the standard situation where admissible controls are those bounded above
and below by appropriate constants $a$ and $b$, that is
\begin{equation}\label{Uad:standard}
\Uad=\{u\in L^2(\O) \,:\,a\le u\le b \mbox{ a.e.}\}.
\end{equation}
Consider a mesh $\mesh$ of $\O$, that is
a finite partition of $\O$ into polygonal/polyhedral cells such that each cell $K\in\mesh$ is star-shaped
with respect to its centroid $\overline{\x}_K$.
Denote the size of this mesh by $h=\max_{K\in\mesh}{\rm diam}(K)$.
The discrete space $\Uh$ is then defined as the
space of piecewise constant functions on this partition:
\begin{equation}\label{Uh:standard}
\Uh=\{v:\O\to \R\,:\,\forall K\in\mesh\,,\;v_{|K}\mbox{ is a constant}\}.
\end{equation}
These choices \eqref{Uad:standard} and \eqref{Uh:standard} of $\Uad$ and $\Uh$ satisfy
\eqref{stab.proj}.
Owing to Remark \ref{rates.control}, for low-order methods such as
conforming and non-conforming FE or MFD schemes, under standard regularity
assumptions the estimate \eqref{est.basic.u}
provides an $\mathcal O(h)$ convergence rate on $\norm{\bu-\bu_h}{}$. Given that $\bu_h$ is piecewise constant, this is optimal. However, using post-processed
controls and following the ideas of \cite{CMAR}, we show that we can obtain
a super-convergence result for the control.

\medskip

The projection operators $\Proj:L^1(\O)\to \Uh$ (orthogonal projection on piecewise constant functions on $\mesh$) and $P_{[a,b]}:\R\to [a,b]$ are defined as
\[
\forall v\in L^1(\O)\,,\;\forall K\in\mesh\,,\quad
(\Proj v)_{|K} :=\dashint_K v(\x)d\x
\]
and
\[
\forall s\in\R\,,\quad P_{[a,b]}(s) := \min(b, \max(a, s)).
\]
 
We make the following assumptions which are discussed, along with the post-processing,
in Section \ref{sec:disc.pp}.

\begin{itemize}
\item[\assum{1}][\emph{Approximation and interpolation errors}]
For each $w\in H^2(\O)$, there exists $w_\mesh\in L^2(\O)$ such that:\\
i) If $w \in H^2(\O) \cap H^1_0(\O)$ solves $-\div(A\nabla w)=g\in H^1(\O)$, and $w_\disc$ is the solution to the corresponding GS, then
\be\label{state:scv}
\norm{\Pi_\disc w_\disc-w_\mesh}{}  \lesssim  h^2\norm{g}{H^1(\O)}.
\ee
ii) For any $w\in H^2(\O)$, it holds
\be\label{prop.M.1}
\begin{aligned}
\forall v\in X_{\disc,0}\,,\;
\big|(w-w_\mesh,\Pi_\disc v_\disc)\big|
\lesssim   h^2\norm{w}{H^2(\O)}\norm{\Pi_\disc v_\disc}{}
\end{aligned}
\ee
and
\be\label{prop.M.2}
\norm{\Proj(w-w_\mesh)}{}\lesssim   h^2\norm{w}{H^2(\O)}.
\ee
\item[\assum{2}] The estimate $ \norm{\Pi_\disc v_\disc-\Proj(\Pi_\disc v_\disc)}{}  \lesssim  h\norm{\nabla_\disc v_\disc}{}$ holds for any $v_\disc\in X_{\disc,0}$.
\item[\assum{3}] [\emph{Discrete Sobolev imbedding}] For all $v\in X_{\disc,0}$, it holds
\[
\norm{\Pi_\disc v_\disc}{L^{2^*}(\O)} \lesssim   \norm{\nabla_\disc v_\disc}{},
\]
where $2^*$ is a Sobolev exponent of 2, that is, $2^*\in [2,\infty)$ if $n=2$, and
$2^*=\frac{2n}{n-2}$ if $n\ge 3$.
\end{itemize}

Let  
\[
\mesh_2=\{K\in\mesh\,:\,\mbox{$\bu=a$ a.e. on $K$, or $\bu=b$ a.e. on $K$,
 or $a<\bu<b$ a.e. on $K$}\}
\]
be the set of fully active or fully inactive cells, and
$\mesh_1=\mesh\setminus\mesh_2$ be the set of cells where $\bu$ takes on the value $a$ (resp.\ $b$) as well as values greater than $a$ (resp.\ lower than $b$).
For $i=1,2$, we let $\O_{i,\mesh}={\rm int}(\cup_{K\in\mesh_i}\overline{K})$.
The space $W^{1,\infty}(\mesh_1)$ is the usual broken Sobolev space,
endowed with its broken norm. Our last assumption is:
\begin{itemize}
\item[\assum{4}] $|\O_{1,\mesh}|\lesssim h$ and $\bu_{|\O_{1,\mesh}}\in W^{1,\infty}(\mesh_1)$, where $|\cdot|$ denotes the Lebesgue measure in $\R^n$.
\end{itemize}

\medskip

From \eqref{opt_cont} and \eqref{opt_discrete}, following the reasoning
in \cite[Theorem 2.28]{tf}, the following pointwise relations can be obtained: for a.e.\ $x \in \O$,
\be
\label{Projection1}
 \begin{aligned}
 \bu(\x)&=P_{[a,b]}\left(\ud(\x)-\frac{1}{\alpha}\bp(\x)\right)\,, \\
 \bu_{h}(\x)&=P_{[a,b]}\left(\Proj\left(\ud(\x) -\frac{1}{\alpha}\Pi_\disc \bp_\disc(\x)\right) \right). 
 \end{aligned}
\ee
Assuming $\bp \in H^2(\O)$ (see Theorem \ref{thm.superconvergence}) and letting $\bp_\mesh$ be defined as in \assum{1}, the post-processed continuous and discrete controls are then defined by
 \begin{equation} 
	\begin{aligned}
 \tu(\x)={}&P_{[a,b]}\left(\Proj \ud(\x)-\frac{1}{\alpha}\bp_\mesh(\x) \right),\\
	\tu_{h}(\x)={}&P_{[a,b]}\left(\Proj\ud(\x)-\frac{1}{\alpha} \Pi_\disc \bp_\disc(\x)\right).
\end{aligned}
\label{Projection2}
 \end{equation}

For $K\in\mesh$, let $\rho_K=\max\{r>0\,:\,B(\overline{\x}_K,r)\subset K\}$
be the maximal radius of balls centred at $\overline{\x}_K$ and included in $K$.
Assume that the mesh is regular, in the sense that there exists $\eta>0$ such that
\be\label{reg:mesh}
\forall K\in\mesh\,,\; \eta\ge \frac{{\rm diam}(K)}{\rho_K}.
\ee
We use the following extension of the notation \eqref{notation:lesssim}:
\[
\begin{aligned}
&X\lesssim_{\eta} Y\mbox{ means that }X\le CY\mbox{ for some $C$ depending}\\
&\mbox{only on $\O$, $A$, an upper bound of $C_\disc$, and $\eta$}.
\end{aligned}
\]

\noindent We now state our two main super-convergence results.

\begin{theorem}[Super-convergence for post-processed controls I] \label{thm.superconvergence}
Let $\disc$ be a GD and $\mesh$ be a mesh.
Assume that 
\begin{itemize}
\item $\Uad$ and $\Uh$ are given by \eqref{Uad:standard} and \eqref{Uh:standard},
\item \assum{1}--\assum{4} hold,
\item $\ud$, $\by$ and $\bp$ belong to $H^2(\O)$,
\item $\yd$ and $f$ belong to $H^1(\O)$,
\end{itemize}
and let $\tu$, $\tu_h$ be the post-processed controls defined by \eqref{Projection2}.
Then there exists $C$ depending only on $\alpha$ such that
\be\label{eq:supercv}
	\norm{\tu-\tu_{h}}{} \lesssim_{\eta}Ch^{2-\frac{1}{2^*}}\norm{\bu}{W^{1,\infty}(\mesh_1)}
+Ch^2\mathcal F(a,b,\yd,\ud,f,\by,\bp),
\ee
where
\begin{align*}
\mathcal F(a,b,\yd,\ud,f,\by,\bp)={}&\minmod(a,b)+\norm{\yd}{H^1(\O)}+  \norm{\ud}{H^2(\O)}+\norm{f}{H^1(\O)}\\
&+\norm{\by}{H^2(\O)}+\norm{\bp}{H^2(\O)}
\end{align*}
with $\minmod(a,b)=0$ if $ab\le 0$ and $\minmod(a,b)=\min(|a|,|b|)$ otherwise.
\end{theorem}

\begin{theorem}[Super-convergence for post-processed controls II] \label{thm.fullsuperconvergence}
Let the assumptions and notations of Theorem \ref{thm.superconvergence} hold, except \assum{3} which is replaced
by: 
\be\label{Linfty.est}
\begin{aligned}
&\mbox{there exists $\delta>0$ such that, for any $F\in L^2(\O)$,}\\
&\mbox{the solution $\psi_D$ to \eqref{base.GS} satisfies $\norm{\Pi_D \psi_D}{L^\infty(\O)}
\le \delta \norm{F}{}$.}
\end{aligned}
\ee
Then there exists $C$ depending only on $\alpha$ and $\delta$ such that
\be \label{eq:fullsupercv}
\norm{\tu-\tu_h}{}\lesssim_{\eta} C h^2\left[\norm{\bu}{W^{1,\infty}(\mesh_1)}+\mathcal F(a,b,\yd,\ud,f,\by,\bp)\right].
\ee
\end{theorem}

\begin{remark}
The estimates \eqref{eq:supercv} and \eqref{eq:fullsupercv}
also hold if we replace the two terms $\Proj\ud$ in \eqref{Projection2}
with $\ud$.
\end{remark}

The super-convergence of the state and adjoint variables follow easily.

\begin{corollary}[Super-convergence for the state and adjoint variables] \label{cor.superconvergence}
 Let $(\by,\bp)$ \\ and $(\by_\disc,\bp_\disc)$ be the solutions to \eqref{state_cont}--\eqref{adj_cont} and  \eqref{discrete_state}--\eqref{discrete_adjoint}. Under the assumptions of Theorem \ref{thm.superconvergence}, the following error estimates hold, with $C$ depending only on $\alpha$:
\begin{align}
\label{eq_supercv.y}
	\norm{\by_\mesh-\Pi_\disc \by_\disc}{}\lesssim_{\eta}{}&
		Ch^{r}\norm{\bu}{W^{1,\infty}(\mesh_1)}+Ch^2 \mathcal F(a,b,\yd,\ud,f,\by,\bp),\\
\label{eq_supercv.p}
		\norm{\bp_\mesh-\Pi_\disc \bp_\disc}{}\lesssim_{\eta}{}&
				C h^{r}\norm{\bu}{W^{1,\infty}(\mesh_1)}
				+Ch^2\mathcal F(a,b,\yd,\ud,f,\by,\bp),
\end{align}
where $\by_\mesh$ and $\bp_\mesh$ are defined as in \assum{1}, and
$r=2-\frac{1}{2^*}$.

Under the assumptions of Theorem \ref{thm.fullsuperconvergence}, \eqref{eq_supercv.y} and \eqref{eq_supercv.p}
hold with $r=2$ and $C$ depending only $\alpha$ and $\delta$.
\end{corollary}

Even for very classical schemes, the $L^\infty$ estimate \eqref{Linfty.est} is only known under
restrictive assumptions on the mesh. For example, for conforming and non-conforming $\mathbb{P}_1$
finite elements, it requires the quasi-uniformity of the mesh \cite{LGRN_NCFEM},
which prevents considering local refinements widely used in practical applications.
The scope of Theorem \ref{thm.fullsuperconvergence} is therefore limited in that sense,
but it nonetheless extends to various methods (see e.g. Corollary \ref{cor:nc.hmm})
the super-convergence established in \cite{CMAR} for conforming $\mathbb{P}_1$ finite elements.

On the contrary, Theorem  \ref{thm.superconvergence} holds under much less restrictive assumptions
(see below for a discussion of \assum{1}--\assum{4}), and applies seamlessly to locally
refined meshes, for essentially all numerical methods currently covered by the GDM.
It is also useful to notice that Theorem \ref{thm.superconvergence} \emph{nearly}
provides an $h^2$ convergence rate.
If $n=2$, the Sobolev exponent $2^*$ can be any finite number. In that case, 
\eqref{eq:supercv}, \eqref{eq_supercv.y} and \eqref{eq_supercv.p} are $\mathcal O(h^{2-\epsilon})$
estimates, for any $\epsilon>0$. If $n=3$, the estimates are of order $\mathcal O(h^{11/6})$.
In each case, as noticed in Section \ref{sec:schemes}, these rates are numerically very
close to a full $\mathcal O(h^2)$ convergence rate.

\subsubsection{Discussion on \assum{1}--\assum{4} and post-processings}\label{sec:disc.pp}

To discuss \assum{1}, \assum{2} and the post-processing
choices \eqref{Projection2}, let us consider two situations depending on the
nature of $\Pi_\disc$. This nature drives the choices of $w_\mesh$,
to ensure that the super-convergence result \eqref{state:scv} holds.

\medskip

\paragraph{\sc $\Pi_\disc$ is a piecewise linear reconstruction}

We consider here the case where $\Pi_\disc v_\disc$ is piecewise linear on $\mesh$ for all $v_\disc\in X_{\disc,0}$. 
Then a super-convergence result \eqref{state:scv} usually holds with $w_\mesh=w$ (and even
$\norm{g}{}$ instead of $\norm{g}{H^1(\O)}$).
This is for example well-known for conforming and non-conforming $\mathbb{P}_1$
FE.
In that case, \eqref{prop.M.1} and \eqref{prop.M.2} are trivially satisfied.

Assumption \assum{2} then follows from a simple Taylor expansion
if $\nabla_\disc v_\disc$ is the classical broken gradient (i.e. the gradient
of $\Pi_\disc v_\disc$ in each cell). This is again the case for conforming and non-conforming
$\mathbb{P}_1$ FE.

The post-processing \eqref{Projection2} of $\bu$ then solely consists in projecting
$\ud$ on piecewise constant functions. In particular, if $\ud$ is already piecewise
constant on the mesh, then $\tu=\bu$.

\medskip

\paragraph{\sc $\Pi_\disc$ is a piecewise constant reconstruction}

We consider that $\Pi_\disc v_\disc$ is piecewise constant on $\mesh$ for all $v_\disc\in X_{\disc,0}$. 
Then the super-convergence \eqref{state:scv} requires to project the exact
solution on piecewise constant functions on the mesh. This is usually done
by setting $w_\mesh(\x)=w(\overline{\x}_K)$ for all $\x\in K$ and all $K\in\mesh$.
This super-convergence result is well-known for hMFD and nodal MFD schemes (see \cite{jd_nn,mimeticfdm}).

In that case, Property \eqref{prop.M.2} follows (with $\lesssim$ replaced with
$\lesssim_\eta$) from the classical approximation result \eqref{approx.xK}.
Using the orthogonality property of $\Proj$, \eqref{prop.M.1}
is then proved by writing
\begin{align*}
|(w-w_\mesh,\Pi_\disc v_\disc)|=|(\Proj (w-w_\mesh),\Pi_\disc v_\disc)|
&\le \norm{\Proj (w-w_\mesh)}{}\norm{\Pi_\disc v_\disc}{}\\
&\lesssim_{\eta} h^2 \norm{w}{H^2(\O)}\norm{\Pi_\disc v_\disc}{}.
\end{align*}

For a piecewise constant reconstruction, \assum{2} is trivial since 
$\Pi_\disc v_\disc=\Proj(\Pi_\disc v_\disc)$.

\medskip

\paragraph{\sc Assumptions \assum{3} and \assum{4}}

Using the discrete functional analysis tools of \cite[Chapter 8]{koala}
the discrete Sobolev embedding \assum{3} is rather straightforward for all methods that
fit in the GDM. This includes conforming and non-conforming
$\mathbb{P}_1$ schemes as well as MFD schemes.

Assumption \assum{4} is identical to the assumption (A3) in \cite{CMAR}.
Let $R$ be the region where the bounds $a$ and $b$ pass from active to inactive, i.e.
where $\ud-\alpha^{-1}\bp$ crosses these bounds.
If $R$ is of co-dimension 1, which is a rather natural situation, then
the condition $|\O_{1,\mesh}|\lesssim h$ holds.

The $W^{1,\infty}$ regularity on $\bu$ mentioned in \assum{4} can be established
in a number of situations.
It holds, for example, if $\O$ is a bounded open subset of class $C^{1,1}$, the coefficients of $A$ belong to $C^{0,1}(\bar{\O})$, $\ud \in W^{1,\infty}(\O)$ and $\yd \in L^q(\O)$ for some $q>n$. Indeed,
under these assumptions, \cite[Theorem 2.4.2.5]{grisvard} ensures that
the state and adjoint equations admit unique solutions in
$H^1_0(\O)\cap W^{2,q}(\O)\subset W^{1,\infty}(\O)$. The projection formula 
\eqref{Projection1} then shows that $\bu$ inherits this Lipschitz continuity property over $\O$.
This also holds if $\O$ has corners but adequate symmetries (that preserve the $W^{2,q}(\O)$
regularity).

Assumption \assum{4} actually does not require the full $W^{1,\infty}$ regularity of $\bu$,
only this regularity on a neighbourhood of $R$. Considering a generic open set $\O$ with Lipschitz (but not necessarily smooth)
boundary, \cite[Theorem 7.3]{stam65} ensures that $\bp$ is continuous. If $\ud$ is continuous
and $a<(\ud)_{|\partial\O}<b$, then $\ud-\alpha^{-1}\bp$ does not cross the levels
$a$ and $b$ close to $\partial\O$, which means that $R$
is a compact set inside $\O$. The Lipschitz regularity of $\bu$ then follows, under the same
assumptions on $A$, $\ud$ and $\yd$ as above, from local regularity results (internal to $\O$),
without assuming that the boundary of $\O$ is $C^{1,1}$.

In all these cases, we also notice that, although the mesh $\mesh$ depends on $h$,
the norm $\norm{\bu}{W^{1,\infty}(\mesh_1)}$ remains bounded independently on $h\to 0$.
Indeed, this norm is bounded by a Lipschitz constant of $\bu$ on a neighbourhood of $R$.

\subsubsection{Application to non-conforming $\mathbb{P}_1$ and hMFD}

Our generic results on the GDM apply to all methods covered by this framework.
In particular, as mentioned in Section \ref{sec:schemes}, to non-conforming $\mathbb{P}_1$ finite
elements and hMFD methods. We state here a corollary of the super-convergence results on the control (Theorems \ref{thm.superconvergence} and \ref{thm.fullsuperconvergence}) for these two methods. We could as 
easily state obvious consequence for these two schemes of Theorem \ref{theorem.control.DB}, Proposition
\ref{theorem.state.adj.DB} and Corollary \ref{cor.superconvergence}.

\begin{corollary}[Super-convergence of the control for nc$\mathbb{P}_1$ and hMFD schemes]\label{cor:nc.hmm}
~

Assume that $\O$ is convex and $A$ is Lipschiz-continu\-ous.
Let $\mesh$ be a mesh in the sense of \cite[Definition 2.21]{DEH15}, with centers at the
centers of mass of the cells.
Assume $\Uad$ and $\Uh$ are given by \eqref{Uad:standard} and \eqref{Uh:standard},
\assum{4} holds, $\ud\in H^2(\O)$ and that $(\yd,f)\in H^1(\O)$.

We consider either one of the following schemes, as described in Section \ref{sec:schemes},
with associated post-processed controls (here, $(\by_h,\bp_h,\bu_h)$ is the solution to
the scheme for the control problem):
\begin{itemize}
\item \textnormal{nc}$\mathbb{P}_1/\mathbb{P}_0$ scheme: $\eta$ satisfies \eqref{reg:mesh}, 
$\tu=P_{[a,b]}(\Proj \ud-\alpha^{-1}\bp)$, and $\tu_h=P_{[a,b]}(\Proj\ud-\alpha^{-1}\bp_h)$.
\item $\textnormal{hMFD}$ schemes: $\eta$ is an upper bound of $\theta_\mesh$ defined by \cite[Eq. (2.27)]{DEH15}
and, for all $K\in\mesh$, 
\[
\tu_{|K}=P_{[a,b]}\left(\dashint_K\ud -\alpha^{-1} \bp(\overline{\x}_K)\right)
\;\mbox{ and }\;
(\tu_h)_{|K}=P_{[a,b]}\left(\dashint_K \ud -\alpha^{-1}(\bp_h)_K\right).
\]
\end{itemize}
Then there exists $C$ depending only on $\O$, $A$, $\alpha$, $a$, $b$, $\bu$, $\ud$, $\yd$, $f$
and $\eta$ such that
\be\label{eq:supercv.nc.hmm}
	\norm{\tu-\tu_{h}}{} \le Ch^{2-\frac{1}{2^*}}.
\ee
Moreover, if $\chi\ge \max_{K\in\mesh}\frac{h^n}{|K|}$, then there exists 
$C$ depending only on $\O$, $A$, $\alpha$, $a$, $b$, $\bu$, $\ud$, $\yd$, $f$, $\eta$ and $\chi$
such that
\be\label{eq:supercv.nc.hmm.full}
	\norm{\tu-\tu_{h}}{} \le Ch^2.
\ee
\end{corollary}

\begin{remark}\label{conformingP1}
The conforming $\mathbb{P}_1$ FE method is a GDM for the gradient discretisation
defined by $\disc=(V_h,{\rm Id},\nabla)$, where $V_h$ is the conforming $\mathbb{P}_1$ space on
the considered mesh. Then, $W_\disc\equiv 0$ and $S_\disc$ is bounded above by the interpolation
error of the $\mathbb{P}_1$ method. For this gradient discretisation method, Theorems \ref{theorem.control.DB} and \ref{thm.fullsuperconvergence}
provide, respectively, $\mathcal O(h)$ error estimates on the control and $\mathcal O(h^2)$
error estimates on the post-processed controls (under a quasi-uniformity assumption on the sequence
of meshes). These rates are the same already proved in \cite{CMAR}. For \textnormal{nc}$\mathbb{P}_1$ FE method, the estimate \eqref{eq:supercv.nc.hmm.full} provides quadratic rate of convergence in a similar way as for conforming $\mathbb{P}_1$ method.
\end{remark}

\begin{proof}[Proof of Corollary \ref{cor:nc.hmm}]
\cite[Sections 3.2.1 and 3.6.1]{DEH15} presents a description of the GDs corresponding
to the nc$\mathbb{P}_1$ and hMFD schemes (the latter is seen as a GS
through its identification as a hybrid mimetic mixed method, see \cite{dro-10-uni,dro-12-gra};
the corresponding GD is recalled in Section \ref{sec:LinftyHMM}, Appendix).

Using these gradient discretisations, \eqref{eq:supercv.nc.hmm}
follows from Theorem \ref{thm.superconvergence} if we can prove that
\assum{1}--\assum{3} hold, for a proper choice of operator $w\mapsto w_\mesh$.

Note that our assumptions on $\O$ and $A$ ensure that the state
(and thus adjoint) equations satisfy the elliptic regularity: if the source terms
are in $L^2(\O)$ then the solutions belong to $H^2(\O)$.

For the nc$\mathbb{P}_1$ scheme, recall that $w_\mesh=w$ and the superconvergence result \eqref{state:scv} is known under the elliptic regularity. Also, $\Pi_\disc w_\disc$
	is simply the solution $w_h$ to the scheme. Properties \eqref{prop.M.1} and
\eqref{prop.M.2} are obvious since $w-w_\mesh=0$. This proves \assum{1}.
Assumption \assum{2} follows easily from a Taylor expansion
since $\nabla_\disc v_\disc$ is the broken gradient of $\Pi_\disc v_\disc$.
Assumption \assum{3} follows from \cite[Proposition 5.4]{DD15}, by noticing that
for piecewise polynomial functions that match at the face centroids, the
discrete $\norm{\cdot}{1,2,h}$ norm in \cite{DD15} boils down to the $L^2(\O)^n$ norm
of the broken gradient.

We now consider the hMFD scheme, for which we let $(w_\mesh)_{|K}=w(\overline{\x}_K)$ for
all $K\in\mesh$. The super-convergence result of \assum{1}-i) is proved in,
e.g., \cite{bre-05-con,jd_nn}. 
As mentioned in Section \ref{sec:disc.pp}, Properties \eqref{prop.M.1} and
\eqref{prop.M.2} follow from \eqref{approx.xK};
\assum{2} is trivially true, and \assum{3} follows from the discrete functional
analysis results of \cite[Lemma 8.15 and Lemma 13.11]{koala}. 

\medskip

The full super-convergence result \eqref{eq:supercv.nc.hmm.full} follows from
Theorem \ref{thm.fullsuperconvergence} if we can establish the $L^\infty$ bound
\eqref{Linfty.est} under the assumption that $\chi$ is bounded -- i.e. the
mesh is quasi-uniform. This $L^\infty$ bound is known for the nc$\mathbb{P}_1$
finite element method \cite{LGRN_NCFEM}, and is proved in Theorem \ref{th:Linfty} for the HMM method.
\end{proof}

\section{Proof of the main results}\label{sec:proofs}

\subsection{Proof of the basic error estimates}

Let us start with a straightforward stability result, which will be useful for the analysis.

\begin{proposition}[Stability of gradient schemes]\label{prop.stab}
Let $\underline{a}$ be a coercivity constant of $A$.
If $\psi_\disc$ is the solution to the gradient scheme \eqref{base.GS}, then
\be\label{stab.grad}
\norm{\nabla_\disc \psi_\disc}{}
\le \frac{C_\disc}{\underline{a}}\norm{F}{} \quad\mbox{ and }\quad\norm{\Pi_\disc \psi_\disc}{}\le 
\frac{C_\disc^2}{\underline{a}}\norm{F}{}.
\ee
\end{proposition}

\begin{proof}
Choose $w_\disc=\psi_\disc$ in \eqref{base.GS} and use the
definition of $C_\disc$ to write
\[
\underline{a}
\norm{\nabla_\disc \psi_\disc}{}^2
\le \norm{F}{}\norm{\Pi_\disc \psi_\disc}{}
\le C_\disc \norm{F}{}\norm{\nabla_\disc \psi_\disc}{}.
\]
The proof of first inequality in \eqref{stab.grad} is complete. The second estimate follows from the
definition of $C_\disc$. \end{proof}


We can now prove Theorem \ref{theorem.control.DB}.
The technique used here is an adaptation of classical ideas used (e.g., for the error-analysis of finite-element based discretisations) to the gradient discretisation method.

In this proof, define the scaled norm $\NORM{\cdot}$ and projection error $E_h$ by
\[
\forall W\in L^2(\O)\,,\;\NORM{W}=\sqrt{\alpha}\norm{W}{}
\mbox{ and } E_h(W)=\NORM{W-\Prh W}.
\]
To establish the error estimates, we need the following auxiliary discrete
problem: seek $(y_{\disc}(\bu), p_{\disc}(\bu) )\in X_{\disc,0} \times X_{\disc,0} $ such that
\begin{subequations} \label{discrete_aux}
	\begin{align}
			a_{\disc}(y_{\disc}(\bu),w_{\disc}) ={}& (f + \bu, \Pi_\disc w_{\disc})&
				\forall \: w_{\disc} \in  X_{\disc,0}, \label{state_aux}\\
		a_{\disc}(w_{\disc},p_{\disc}(\bu)) ={}& (\by-\yd, \Pi_\disc w_{\disc}) &
			  \forall \: w_{\disc} \in  X_{\disc,0}. \label{adj_aux}
	\end{align}
\end{subequations}

\begin{proof}[Proof of Theorem \ref{theorem.control.DB}]
Let $P_{\disc,\alpha}(\bu)=\alpha^{-1}\Pi_\disc p_{\disc}(\bu)$,
$\bP_{\disc,\alpha}=\alpha^{-1}\Pi_\disc \bp_{\disc}$, and  $\bP_{\alpha}=\alpha^{-1} \bp$.

Since $\bu_h \in \Uadh \subset \Uad$, from the optimality condition \eqref{opt_cont}, 
\begin{align} \label{inter1}
&-\alpha(\bP_{\alpha}+\bu-\ud,\bu-\bu_h)\ge 0.
\end{align}
By \eqref{stab.proj}, we have $\Prh\bu \in \Uadh$ and therefore
a use of the discrete optimality condition (see \eqref{opt_discrete}) yields
\begin{align}
 \alpha(\bP_{\disc,\alpha}+\bu_h-\ud,\bu - \bu_h) ={}&\alpha(\bP_{\disc,\alpha} +\bu_h-\ud,\bu- \Prh\bu)\nonumber\\
 & +\alpha (\bP_{\disc,\alpha} +\bu_h-\ud,\Prh\bu -\bu_h)\nonumber\\
\ge{}& \alpha(\bP_{\disc,\alpha}+\bu_h-\ud,\bu- \Prh\bu). \label{inter2}
\end{align}
An addition of \eqref{inter1} and \eqref{inter2} yields
\begin{align}
\NORM{\bu-\bu_h}^2  \le{}& -\alpha(\bP_{\disc,\alpha}+\bu_h-\ud,\bu- \Prh\bu)
+\alpha(\bP_{\disc,\alpha} - \bP_{\alpha},\bu-\bu_h)
\nonumber\\
={}& -\alpha(\bP_{\disc,\alpha} +\bu_h-\ud,\bu-\Prh\bu)
-\alpha(\bP_{\alpha}-P_{\disc,\alpha}(\bu),\bu-\bu_h)\nonumber\\
& +\alpha(\bP_{\disc,\alpha}-P_{\disc,\alpha}(\bu),\bu-\bu_h).
\label{inter3}
\end{align}
The first term in the right-hand side of \eqref{inter3} is recast now. By orthogonality property of $\Prh$ we have 
$(\bu_h-\Prh \ud,\bu-\Prh\bu)=0$ and 
$(\Prh\bP_{\alpha},\bu-\Prh\bu)=0$.
Therefore,
\begin{align}
-\alpha({}&\bP_{\disc,\alpha}+\bu_h-\ud,\bu-\Prh\bu)\nonumber\\
={}&-\alpha(\bP_{\alpha},\bu-\Prh\bu) +\alpha(\bP_{\alpha}-\bP_{\disc,\alpha},\bu-\Prh\bu) -\alpha(\Prh \ud-\ud,\bu-\Prh\bu) \nonumber \\
={}&-\alpha(\bP_{\alpha}-\Prh\bP_{\alpha},\bu-\Prh\bu)+\alpha(\bP_{\alpha}-P_{\disc,\alpha}(\bu),\bu-\Prh\bu)\nonumber \\ 
& +\alpha(P_{\disc,\alpha}(\bu)-\bP_{\disc,\alpha},\bu-\Prh\bu)
+\alpha(\ud-\Prh \ud,\bu-\Prh\bu).
\label{inter4}
\end{align}
Let us turn to the third term in the right-hand side of \eqref{inter3}.
From \eqref{discrete_adjoint} and \eqref{adj_aux}, for all $w_\disc \in X_{\disc,0}$,
\be\label{eq:1}
a_\disc(w_\disc,\bp_\disc- p_\disc(\bu)) = (\Pi_\disc \by_\disc- \by, \Pi_\disc w_\disc). 
\ee
We also have, by \eqref{discrete_state} and \eqref{state_aux},
\begin{align}
a_\disc(\by_\disc- y_\disc(\bu), w_\disc) =& (\bu_h-\bu, \Pi_\disc w_\disc).
\label{eq:2}
\end{align}
A use of symmetry of $a_\disc$,
a choice of $w_\disc = \by_\disc-y_\disc(\bu)$ in \eqref{eq:1} and
$w_\disc=\bp_\disc-p_\disc(\bu)$ in \eqref{eq:2} gives
an expression for the third term on the right hand side of \eqref{inter3} as
\begin{align}
\alpha(\bP_{\disc,\alpha}-P_{\disc,\alpha}(\bu),\bu-\bu_h)
 ={}&- (\Pi_\disc \by_\disc- \by, \Pi_\disc \by_\disc- \Pi_\disc y_\disc(\bu)) \nonumber \\
={}& (\by - \Pi_\disc y_\disc(\bu), \Pi_\disc \by_\disc- \Pi_\disc y_\disc(\bu)) \nonumber \\
& - \norm{\Pi_\disc \by_\disc-  \Pi_\disc y_\disc(\bu)}{}^2.
\label{inter5}
\end{align}
A substitution of \eqref{inter4} and \eqref{inter5} in \eqref{inter3} yields
\begin{align}
&\NORM{\bu - \bu_h}^2+ \norm{\Pi_\disc \by_\disc-  \Pi_\disc y_\disc(\bu)}{}^2
\nonumber \\ 
&\quad\le-\alpha(\bP_{\alpha}-\Prh\bP_{\alpha},\bu-\Prh\bu)
+\alpha(\bP_{\alpha}-P_{\disc,\alpha}(\bu),\bu-\Prh\bu)\nonumber\\
&\qquad+\alpha(P_{\disc,\alpha}(\bu)-\bP_{\disc,\alpha},\bu-\Prh\bu)
 +\alpha(\ud-\Prh \ud,\bu-\Prh\bu) \nonumber\\
&\qquad
-\alpha(\bP_{\alpha}-P_{\disc,\alpha}(\bu),\bu-\bu_h)
 +  (\by - \Pi_\disc y_\disc(\bu) , \Pi_\disc \by_\disc- \Pi_\disc y_\disc(\bu))
 \nonumber \\ 
& \quad=: T_1 + T_2+ T_3 + T_4+T_5+T_6.
\label{big.last.estimate}
\end{align}
We now estimate each term $T_i$, $i=1,\ldots,6$.
By Cauchy-Schwarz inequality we have 
\be\label{est.T1}
T_1\le E_h(\bP_{\alpha})E_h(\bu).
\ee
Equation \eqref{adj_aux} shows that $p_\disc(\bu)$ is the
solution of the GS corresponding to the adjoint
problem \eqref{adj_cont}, whose solution is $\bp$. Therefore, by Theorem \ref{th:error.est.PDE},
\begin{equation}
\NORM{\bP_{\alpha}-P_{\disc,\alpha}(\bu)}
=\frac{1}{\sqrt{\alpha}}\norm{\bp-\Pi_\disc p_\disc(\bu)}{}\lesssim \frac{1}{\sqrt{\alpha}}\WS_\disc(\bp).
\label{est.bPPdisc}
\end{equation}
Hence, using the Cauchy--Schwarz inequality,
\be
T_2\lesssim \frac{1}{\sqrt{\alpha}}E_h(\bu)\WS_\disc(\bp).
\label{est.T2}
\ee
Let us turn to $T_3$. By writing the difference of \eqref{adj_aux}
and \eqref{discrete_adjoint} we see that $p_\disc(\bu)-\bp_\disc$
is the solution to the GS \eqref{base.GS} with
source term $F=\by-\Pi_\disc\by_\disc$. Hence, using Proposition
\ref{prop.stab}, we find that
\begin{align*}
\NORM{P_{\disc,\alpha}(\bu)-\bP_{\disc,\alpha}}
={}&\frac{1}{\sqrt{\alpha}}
\norm{\Pi_\disc p_\disc(\bu)-\Pi_\disc \bp_\disc}{}\\
\lesssim{}&\frac{1}{\sqrt{\alpha}} \norm{\by-\Pi_\disc\by_\disc}{}\\
\lesssim{}&\frac{1}{\sqrt{\alpha}} \norm{\by-\Pi_\disc y_\disc(\bu)}{}
+ \frac{1}{\sqrt{\alpha}}\norm{\Pi_\disc y_\disc(\bu)-\Pi_\disc \by_\disc}{}.
\end{align*}
A use of Theorem \ref{th:error.est.PDE} with $\psi=\by$ to bound the first term in the above expression yields, by Young's inequality, 
\be\label{est.T3}
T_3\le{} \frac{\cter{cst:T3}}{\sqrt{\alpha}}E_h(\bu)\WS_\disc(\by)+\frac{\cter{cst:T3}}{\alpha}E_h(\bu)^2
+\frac{1}{4}\norm{\Pi_\disc y_\disc(\bu)-\Pi_\disc \by_\disc}{}^2,
\ee
where $\ctel{cst:T3}$ only depends on $\O$, $A$ and an upper bound of $C_\disc$.\\
Let us consider $T_4$ now. A use of Cauchy--Schwarz inequality and Young inequality leads to
\be\label{est.T4}
T_4 \le E_h(\bu)E_h(\ud)\le \frac{1}{2} E_h(\bu)^2+\frac{1}{2}E_h(\ud)^2.
\ee
We estimate $T_5$ by using \eqref{est.bPPdisc} and Young's inequality:
\be\label{est.T5}
T_5\le \frac{1}{2}\NORM{\bu- \bu_h}{}^2
+\frac{\cter{cst:gen}}{\alpha}\WS_\disc(\bp)^2,
\ee
where $\ctel{cst:gen}$ only depends on $\O$, $A$ and an upper bound of $C_\disc$.
Finally, to estimate $T_6$ we write, by Theorem \ref{th:error.est.PDE} with $\psi=\by$,
\be\label{est.T6}
\begin{aligned}
T_6\le{}& \cter{cst:gen2}\WS_\disc(\by)\norm{\Pi_\disc\by_\disc-\Pi_\disc y_\disc(\bu)}{}\\
\le{}&\cter{cst:gen2}^2\WS_\disc(\by)^2+
\frac{1}{4}\norm{\Pi_\disc\by_\disc-\Pi_\disc y_\disc(\bu)}{}^2,
\end{aligned}
\ee
with $\ctel{cst:gen2}$ only depending on $\O$, $A$ and an upper bound of $C_\disc$.

We then plug \eqref{est.T1}, \eqref{est.T2}, \eqref{est.T3}, \eqref{est.T4}, \eqref{est.T5} and \eqref{est.T6} into \eqref{big.last.estimate}. A use of Young's inequality and
$\sqrt{\sum_i a_i^2}\le \sum_i a_i$ concludes the proof. \end{proof}

\subsection{Proof of the super-convergence estimates}

The following auxiliary problem will be useful to prove the superconvergence of the control.
For $g\in L^2(\O)$, let $p_\disc^*(g)\in X_{\disc,0}$ solve
\begin{equation} 
a_\disc(w_\disc,p_\disc^*(g))=(\Pi_\disc y_\disc(g)-\yd,\Pi_\disc w_\disc)\quad\forall w_\disc\in X_{\disc,0}, \label{aux2}
\end{equation}
where $y_\disc(g)$ is given by \eqref{state_aux} with $\bu$ replaced by $g$.

Let us recall two approximation properties of $\Proj$. As proved in \cite[Lemma 8.10]{koala},
\begin{equation}\label{approx.PM.1}
\forall\phi\in H^1(\O)\,,\;\norm{\Proj\phi-\phi}{}\lesssim_{\eta} h\norm{\phi}{H^1(\O)}.
\end{equation}
For $K\in\mesh$, let $\overline{\x}_K$ be the centroid (centre of gravity) of $K$.
We have the standard approximation property (see e.g. \cite[Lemma 7.7]{jd_nn} with
$w_K\equiv 1$)
\begin{equation}\label{approx.xK}
\forall K\in\mesh\,,\;\forall \phi\in H^2(K)\,,\;
\norm{\Proj \phi-\phi(\overline{\x}_K)}{L^2(K)}\lesssim_{\eta} {\rm diam}(K)^2 \norm{\phi}{H^2(K)}.
\end{equation}

\begin{proof}[Proof of Theorem \ref{thm.superconvergence}]
Define $\hat{u}$, $\hat{p}$ and $\udh$ a.e.\ on $\O$ by:
for all $K\in\mesh$ and all $\x\in K$,
$\hat{u}(\x) = \bu(\overline{\x}_K)$,  $ \hat{p}(\x)=\bp(\overline{\x}_K)$ and $\udh(\x)=\ud(\overline{\x}_K)$.
From \eqref{Projection2} and the Lipschitz continuity of $P_{[a,b]}$, it follows that
\begin{align}
\norm{\tu-\tu_{h}}{}
\le{}&\alpha^{-1}\norm{ \Pi_\disc \bp_\disc-\bp_\mesh}{}\nonumber\\
\le{}&  \alpha^{-1}\norm{ \bp_\mesh-\Pi_\disc p_\disc^*(\bu)}{} +  \alpha^{-1}\norm{\Pi_\disc p_\disc^*(\bu) -\Pi_\disc p_\disc^*(\hat{u})}{}  \nonumber \\
& + \alpha^{-1}\norm{\Pi_\disc p_\disc^*(\hat{u})-\Pi_\disc \bp_\disc}{}  \nonumber \\
 =:{}&   \alpha^{-1}A_1 + \alpha^{-1}A_2+ \alpha^{-1}A_3.\label{ineq}
\end{align}

\textbf{Step 1}: estimate of $A_1$.

Recalling the equations \eqref{adj_cont}
and \eqref{adj_aux} on $\bp$ and $p_\disc(\bu)$, a use of triangle inequality
and $\assum{1}$-i)  yields
\begin{align}
A_1 & \le  \norm{ \bp_\mesh-\Pi_\disc p_\disc(\bu)}{} + \norm{ \Pi_\disc p_\disc(\bu)-\Pi_\disc p_\disc^*(\bu)}{} \nonumber\\
&\lesssim   h^2\norm{\by -\yd}{H^1(\O)} +\norm{ \Pi_\disc p_\disc(\bu)-\Pi_\disc p_\disc^*(\bu)}{}.
\label{a1.inequality}
\end{align} 

We now estimate the last term in this inequality. 
Subtract \eqref{aux2} with $g=\bu$ from \eqref{adj_aux}, substitute $w_\disc=p_\disc(\bu)-p_\disc^*(\bu)$ , use Cauchy-Schwarz inequality and 
property \eqref{prop.M.1} in $\assum{1}$-ii) to obtain
\begin{align*}
 \norm{\nabla_\disc(p_\disc(\bu)-p_\disc^*(\bu))}{}^2 & \lesssim{}
a_\disc(p_\disc(\bu)-p_\disc^*(\bu),p_\disc(\bu)-p_\disc^*(\bu)) \\
& = (\by- \Pi_\disc y_\disc (\bu), \Pi_\disc(p_\disc(\bu)-p_\disc^*(\bu))  ) \\
& = (\by- \by_\mesh, \Pi_\disc(p_\disc(\bu)-p_\disc^*(\bu))  ) \\
& \qquad + (\by_\mesh- \Pi_\disc y_\disc (\bu), \Pi_\disc(p_\disc(\bu)-p_\disc^*(\bu))  )\\
&\lesssim{}
 h^2\norm{\by}{H^2(\O)} \norm{\Pi_\disc(p_\disc(\bu)-p_\disc^*(\bu))}{}\\ 
& \qquad + \norm{\by_\mesh-\Pi_\disc y_\disc(\bu)}{} \norm{\Pi_\disc(p_\disc(\bu)-p_\disc^*(\bu))}{}.  \end{align*} 

Using the definition of $C_\disc$ and $\assum{1}$-i) leads to
$\norm{\Pi_\disc p_\disc(\bu)-\Pi_\disc p_\disc^*(\bu)}{} \lesssim h^2\norm{\by}{H^2(\O)} + h^2\norm{f+\bu}{H^1(\O)}$. Plugged into \eqref{a1.inequality}, this estimate yields
\begin{align}
A_1 \lesssim  h^2(\norm{\by -\yd}{H^1(\O)}+  \norm{\by}{H^2(\O)}
+\norm{f+\bu}{H^1(\O)} ). \label{a1new.ineq}
\end{align}

\textbf{Step 2}: estimate of $A_2$.

Subtracting the equations  \eqref{aux2} satisfied by $p_\disc^*(\bu)$
and $p_\disc^*(\hat{u})$, for all $v_\disc \in X_{\disc,0}$,
\be \label{subtract_aux2}
a_\disc(v_\disc, p_\disc^*(\bu) -p_\disc^*(\hat{u}))=(\Pi_\disc y_\disc(\bu) -\Pi_\disc y_\disc(\hat{u}),\Pi_\disc v_\disc).
\ee
As a consequence of \eqref{subtract_aux2} and Proposition \ref{prop.stab}, 
\begin{align}
A_2 &=\norm{\Pi_\disc p_\disc^*(\bu) -\Pi_\disc p_\disc^*(\hat{u})}{}
\lesssim \norm{\Pi_\disc y_\disc(\bu) -\Pi_\disc y_\disc(\hat{u})}{}. \label{est.A2}
\end{align}
Choosing $v_\disc=y_\disc(\bu)-y_\disc(\hat{u})$ in \eqref{subtract_aux2}, setting $w_\disc=p_\disc^*(\bu)-p_\disc^*(\hat{u})$, subtracting the equations \eqref{state_aux} satisfied by $y_\disc(\bu)$
and $y_\disc(\hat{u})$, using the orthogonality
property of the projection operator $\Proj$, and invoking \eqref{approx.PM.1}
and $\assum{2}$ gives
\begin{align}
\Vert\Pi_\disc( y_\disc&(\bu)-y_\disc(\hat{u}))\Vert^2 =(\Pi_\disc y_\disc(\bu)-\Pi_\disc y_\disc(\hat{u}),\Pi_\disc y_\disc(\bu)-\Pi_\disc y_\disc(\hat{u})) \nonumber\\
={}&a_\disc(y_\disc(\bu)-y_\disc(\hat{u}),p_\disc^*(\bu)-p_\disc^*(\hat{u}))\nonumber \\
={}& (\bu-\hat{u},\Pi_\disc w_\disc) \nonumber \\
 ={}& (\bu-\Proj\bu,\Pi_\disc w_\disc-\Proj(\Pi_\disc w_\disc))
 + (\Proj\bu-\hat{u},\Pi_\disc w_\disc) \nonumber \\
\lesssim_{\eta}{}& h \norm{\bu}{H^1(\O)}h\norm{\nabla_\disc w_\disc}{}  \nonumber\\
&+ \underbrace{\int_{\O_{1,\mesh}}(\Proj\bu-\hat{u})\Pi_\disc w_\disc\d\x}_{A_{21}}  +
\underbrace{\int_{\O_{2,\mesh}}(\Proj\bu-\hat{u})\Pi_\disc w_\disc\d\x}_{A_{22}}. \label{new.A2}
\end{align}
 Equation \eqref{subtract_aux2} and Proposition \ref{prop.stab} show that
 \be \label{stability_aux2}
\norm{\nabla_\disc w_\disc}{}= \norm{\nabla_\disc(p_\disc^*(\bu)-p_\disc^*(\hat{u}))}{} \lesssim \norm{\Pi_\disc( y_\disc(\bu)-y_\disc(\hat{u}))}{}.
 \ee
Plugging this estimate into \eqref{new.A2} yields
\be
\Vert\Pi_\disc( y_\disc(\bu)-y_\disc(\hat{u}))\Vert^2\lesssim_{\eta}{} h^2 \norm{\bu}{H^1(\O)}\norm{\Pi_\disc( y_\disc(\bu)-y_\disc(\hat{u}))}{} +A_{21}+A_{22}. \label{a21a22.ineq}
\ee
A use of Holder's inequality, $\assum{4}$, $\assum{3}$ and \eqref{stability_aux2} yields
\begin{align}
A_{21} \le{}& \norm{\Proj\bu-\hat{u}}{L^2(\O_{1,\mesh})}\norm{\Pi_\disc w_\disc}{L^2(\O_{1,\mesh})} \nonumber \\
  \le{}& h \norm{\bu}{W^{1,\infty}(\mesh_1)}|\O_{1,\mesh}|^\frac{1}{2} \norm{\Pi_\disc w_\disc}{L^{2^*}(\O)}|\O_{1,\mesh}|^{\frac{1}{2}-\frac{1}{2^*}} \nonumber \\
\lesssim{}&   h^{2-\frac{1}{2^*}}\norm{\bu}{W^{1,\infty}(\mesh_1)}\norm{\nabla_\disc w_\disc}{} \nonumber\\
 \lesssim{}&   h^{2-\frac{1}{2^*}}\norm{\bu}{W^{1,\infty}(\mesh_1)}\norm{\Pi_\disc( y_\disc(\bu)-y_\disc(\hat{u}))}{}. \label{a21.ineq}
\end{align}
Consider now $A_{22}$. 
For any $K\in \mesh_2$, we have $\bu=a$ on $K$, $\bu=b$ on $K$,
or, by \eqref{Projection1}, $\bu=\ud-\alpha^{-1}\bp$.
Hence, $\bu\in H^2(K)$ and, using \eqref{approx.xK}, the definition of $C_\disc$ and \eqref{stability_aux2}, we obtain
 \begin{align}
 A_{22}  {}&\le \norm{\Proj\bu-\hat{u}}{L^2(\O_{2,\mesh})}\norm{\Pi_\disc w_\disc}{} \nonumber\\
{}&\lesssim_{\eta} h^2 \norm{\bu}{H^2(\O_{2,\mesh})}\norm{\Pi_\disc w_\disc}{} \nonumber \\
 {}&\lesssim_{\eta} h^2 \left( \norm{\ud}{H^2(\O_{2,\mesh})} +\alpha^{-1}\norm{\bp}{H^2(\O_{2,\mesh})}\right)\norm{\nabla_\disc w_\disc}{}\nonumber \\
 {}&\lesssim_{\eta} h^2 \left( \norm{\ud}{H^2(\O_{2,\mesh})} +\alpha^{-1}\norm{\bp}{H^2(\O_{2,\mesh})}\right)\norm{\Pi_\disc( y_\disc(\bu)-y_\disc(\hat{u}))}{}. \label{a22.ineq}
 \end{align}
 
Plugging \eqref{a21.ineq} and \eqref{a22.ineq} into \eqref{a21a22.ineq} yields
\begin{align}
\norm{\Pi_\disc( y_\disc(\bu)-y_\disc(\hat{u}))}{}
\lesssim_{\eta}{}&
h^2\norm{\bu}{H^1(\O)}+h^{2-\frac{1}{2^*}}\norm{\bu}{W^{1,\infty}(\mesh_1)} \nonumber\\
&+ h^2 \left( \norm{\ud}{H^2(\O_{2,\mesh})} +\alpha^{-1}\norm{\bp}{H^2(\O_{2,\mesh})}\right). \label{yd.ineq}
\end{align}
Hence, using this in \eqref{est.A2}, we infer
\be
A_{2} \lesssim_{\eta}{} h^{2-\frac{1}{2^*}}\norm{\bu}{W^{1,\infty}(\mesh_1)}
+ h^2(\norm{\bu}{H^1(\O)}+ \alpha^{-1}\norm{\bp}{H^2(\O)}+ \norm{\ud}{H^2(\O)} ). 
\label{a2.ineq}
\ee

\medskip

\textbf{Step 3}: estimate of $A_3$.

Applying twice the stability result of Proposition \ref{prop.stab} (first on
the equation satisfied by $p_\disc^*(\hat{u})-\bp_\disc$, and then
on $y_\disc(\hat{u})-\by_\disc$), we write
\be\label{a3.ineq}
A_3 =\norm{\Pi_\disc p_\disc^*(\hat{u})-\Pi_\disc \bp_\disc}{}  \lesssim \norm{\Pi_\disc y_\disc(\hat{u})-\Pi_\disc \by_\disc}{} 
	 \lesssim \norm{\hat{u}-\bu_h}{}.
\ee
Using the continuous optimality condition \eqref{opt_cont}, as in the proof of \cite[Lemma 3.5]{CMAR} we have, for a.e.\ $\x \in \O$,
\[
\big[\bp(\x)+\alpha(\bu(\x)-\ud(\x))\big]\,\big[v(\x)-\bu(\x)\big] \geq 0 \mbox{   for all   } v \in  \Uad. 
\]
Since $\bu$, $\bp$ and $\bu_h$ are continuous at the centroid $\overline{\x}_K$, we
can choose $\x=\overline{\x}_K$ and $v(\overline{\x}_K)=\bu_h(\overline{\x}_K)(=\bu_h$ on $K$). All the involved
functions being constants over $K$, this gives
\[
 \left(\hat{p}+\alpha(\hat{u}-\udh)\right) \left(\bu_h-\hat{u}\right) \geq 0 \mbox{ on $K$, for all $K\in\mesh$}.
\]
Integrating over $K$ and summing over $K \in \mesh$,
\[
(\hat{p} +\alpha \left( \hat{u}-\udh\right) ,\bu_h-\hat{u}) \geq 0.
\]
Choose $v_h=\hat{u}$ in the discrete optimality condition \eqref{opt_discrete} to obtain
\[(\Pi_\disc\bp_\disc+\alpha(\bu_h-\ud),\hat{u}-\bu_h) \geq 0. \]
Adding the above two inequalities yield
\[(\hat{p}-\Pi_\disc\bp_\disc+\alpha(\hat{u}-\bu_h)+ \alpha(\ud-\udh ),\bu_h-\hat{u}) \geq 0
\]
and thus
\begin{align}
\alpha \norm{\hat{u}-\bu_h}{}^2 \leq{}& (\hat{p}-\Pi_\disc\bp_\disc,\bu_h-\hat{u})+  \alpha(\ud-\udh,\bu_h-\hat{u})  \nonumber \\
={}&(\hat{p}-\bp_\mesh,\bu_h-\hat{u})+(\bp_\mesh-\Pi_\disc p_\disc^*(\hat{u}),\bu_h-\hat{u})\nonumber \\
& +(\Pi_\disc p_\disc^*(\hat{u})-\Pi_\disc\bp_\disc,\bu_h-\hat{u})
 +  \alpha(\ud-\udh,\bu_h-\hat{u})  \nonumber \\
=:{}&M_1+M_2+M_3+M_4.\label{a3new.ineq}
\end{align}
Since $\bu_h-\hat{u}$ is piecewise constant on $\mesh$,
the orthogonality property of $\Proj$, \eqref{approx.xK} and \eqref{prop.M.2}
in $\assum{1}$-ii) lead to
\begin{align}
M_1&= (\hat{p}-\Proj\bp_\mesh,\bu_h-\hat{u}) \nonumber \\
&=  (\hat{p}-\Proj\bp,\bu_h-\hat{u})+ (\Proj(\bp-\bp_\mesh),\bu_h-\hat{u}) \nonumber \\
& \le \norm{\hat{p}-\Proj\bp}{}\norm{\bu_h-\hat{u}}{}+\norm{\Proj(\bp-\bp_\mesh)}{} \norm{\bu_h-\hat{u}}{}\nonumber \\
& \lesssim_{\eta}  h^2\norm{\bp}{H^2(\O)}\norm{\bu_h-\hat{u}}{}. \label{m1.ineq}
\end{align}	 

By Cauchy--Schwarz inequality, triangle inequality and the definitions
of $A_1$ and $A_2$,
\be
M_2\le \norm{\bp_\mesh-\Pi_\disc p_\disc^*(\hat{u})}{}\norm{\bu_h-\hat{u}}{}
\lesssim (A_1+A_2) \norm{\bu_h-\hat{u}}{}.\label{m2.ineq}
\ee
Subtracting the equations \eqref{discrete_state} and \eqref{state_aux} (with
$\hat{u}$ instead of $\bu$) satisfied by $\by_\disc$ and $y_\disc(\hat{u})$, choosing $w_\disc= p_\disc^*(\hat{u})-\bp_\disc$, and using the equations
\eqref{discrete_adjoint} and \eqref{aux2} on $\bp_\disc$ and $p_\disc^*(\hat{u})$,
we find
\begin{align}
M_3	={}&(\Pi_\disc (p_\disc^*(\hat{u})-\bp_\disc),\bu_h-\hat{u})\nonumber\\
={}& a_\disc(\by_\disc-y_\disc(\hat{u}),p_\disc^*(\hat{u})-\bp_\disc) \nonumber\\
={}& (\Pi_\disc(y_\disc(\hat{u})-\by_\disc),\Pi_\disc(\by_\disc-y_\disc(\hat{u}))
\le 0. \label{m3.ineq}
\end{align}
Using the orthogonality property of $\Proj$, \eqref{approx.xK}
yields
\begin{align}
M_4 = \alpha(\ud-\udh,\bu_h-\hat{u})
&= \alpha(\Proj\ud-\udh,\bu_h-\hat{u}) \nonumber \\
      & \lesssim_{\eta}{} \alpha\norm{\Proj\ud-\udh}{}\norm{\bu_h-\hat{u}}{} \nonumber\\
&\lesssim_{\eta}{} \alpha h^2 \norm{\ud}{H^2(\O)}\norm{\bu_h-\hat{u}}{}. \label{m4.ineq}
\end{align} 
A substitution of \eqref{m1.ineq}, \eqref{m2.ineq} (together with the estimates
\eqref{a1new.ineq} and \eqref{a2.ineq} on $A_1$ and $A_2$), \eqref{m3.ineq} and \eqref{m4.ineq} into \eqref{a3new.ineq} yields an estimate on $\norm{\bu_h-\hat{u}}{}$ which, when
plugged into \eqref{a3.ineq}, gives
\be
\begin{aligned}
A_{3} \lesssim{}&   \norm{\bu_h-\hat{u}}{} \\
\lesssim_{\eta}{}& \alpha^{-1}
h^{2-\frac{1}{2^*}}\norm{\bu}{W^{1,\infty}(\mesh_1)}\\
&+\alpha^{-1}h^2\big[\norm{\by -\yd}{H^1(\O)}+  \norm{\by}{H^2(\O)}+
(1+\alpha^{-1})\norm{\bp}{H^2(\O)}\\
&\quad\quad+\norm{f+\bu}{H^1(\O)}+\norm{\bu}{H^1(\O)} +(1+\alpha)\norm{\ud}{H^2(\O)}\big]. \label{a3new1.ineq}
\end{aligned}
\ee

\textbf{Step 4}: conclusion.

It is easy to check that $|P_{[a,b]}(s)|\le \minmod(a,b)+|s|$, where $\minmod$ is
defined in Theorem \ref{thm.superconvergence}.
Hence, by \eqref{Projection1} and Lipschitz continuity of $P_{[a,b]}$,
\begin{align}\label{est.bu.H1}
\norm{\bu}{H^1(\O)}\le{}&\norm{P_{[a,b]}\left(\ud-\alpha^{-1}\bp\right)}{L^2(\O)}
+\norm{\nabla \left(P_{[a,b]}\left(\ud-\alpha^{-1}\bp\right)\right)}{L^2(\O)^n} \nonumber\\
\le{}& \minmod(a,b)|\O|^{1/2}+2\norm{\ud-\alpha^{-1}\bp}{H^1(\O)} \nonumber\\
\le{}& \minmod(a,b)|\O|^{1/2}+2\norm{\ud}{H^1(\O)}+2\alpha^{-1}\norm{\bp}{H^1(\O)}.
\end{align}
Using this inequality and
inserting \eqref{a1new.ineq}, \eqref{a2.ineq} and \eqref{a3new1.ineq} in \eqref{ineq},
the proof of Theorem \ref{thm.superconvergence} is complete.
\end{proof}

\begin{proof}[Proof of Theorem \ref{thm.fullsuperconvergence}]
The proof of this theorem is identical to the proof of Theorem \ref{thm.superconvergence}, except for the estimate 
of $A_{21}$. This estimate is the only source of the $2-\frac{1}{2^*}$ power (instead of 2),
and the only place where we used Assumption \assum{3}, here replaced by \eqref{Linfty.est}.
The estimate of $A_{21}$ using this $L^\infty$-bound assumption is actually rather simple.
Recalling \assum{4} and using \eqref{Linfty.est} on the equation \eqref{subtract_aux2}
satisfied by  $p_\disc^*(\bu)-p_\disc^*(\hat{u})$, we write
	\begin{align}
	A_{21}&=\int_{\O_{1,\mesh}}(\Proj\bu-\hat{u})\big(\Pi_\disc p_\disc^*(\bu)-\Pi_\disc p_\disc^*(\hat{u})\big)\d\x  
\nonumber\\
	&\lesssim \norm{\Proj\bu-\hat{u}}{L^\infty(\O_{1,\mesh})} \norm{\Pi_\disc p_\disc^*(\bu)-\Pi_\disc p_\disc^*(\hat{u})}{L^\infty(\O_{1,\mesh})}|\O_{1,\mesh}|  \nonumber\\
	&\lesssim h^2 \norm{\bu}{W^{1,\infty}(\mesh_1)}\norm{\Pi_\disc p_\disc^*(\bu)-\Pi_\disc p_\disc^*(\hat{u})}{L^\infty(\O)}  \nonumber\\
		&\lesssim h^2 \norm{\bu}{W^{1,\infty}(\mesh_1)}\delta \norm{\Pi_\disc y_\disc(\bu)-\Pi_\disc y_\disc(\hat{u})}{}.
\label{improved.A2}
	\end{align}
The rest of the proof follows from this estimate.
	\end{proof}
 
\subsection{Proof of the error estimates for the state and adjoint variables}

\begin{proof}[Proof of Proposition \ref{theorem.state.adj.DB}] 
Applying the triangle inequality twice, 
\begin{align*}
\norm{\Pi_\disc \by_\disc-\by}{}+\norm{\nabla_\disc \by_\disc-\nabla\by}{}  
\le{}&\norm{\Pi_\disc \by_\disc-\Pi_\disc y_\disc(\bu) }{}+\norm{\Pi_\disc y_\disc(\bu)-\by}{} \\
&+ \norm{\nabla_\disc \by_\disc-\nabla_\disc y_\disc(\bu)}{} +\norm{\nabla_\disc y_\disc(\bu)-\nabla\by}{} .
\end{align*}
The second and last  terms on the right hand side of the above inequality are estimated using Theorem \ref{th:error.est.PDE} as
$$ \norm{\Pi_\disc y_\disc(\bu)-\by}{}+\norm{\nabla_\disc y_\disc(\bu)-\nabla\by}{} \lesssim \WS_\disc(\by) .$$ 
Subtracting \eqref{discrete_state} and \eqref{state_aux}, and using the stability property of
GSs (Proposition \ref{prop.stab}), we obtain
$$ \norm{\Pi_\disc \by_\disc-\Pi_\disc y_\disc(\bu) }{}+\norm{\nabla_\disc \by_\disc-\nabla_\disc y_\disc(\bu)}{} \lesssim \norm{\bu-\bu_h}{}.$$
A combination of the above two results yields the error estimates \eqref{est.basic.y} for the state variable. The error estimate for the adjoint variable can be obtained similarly.
\end{proof}

\begin{proof}[Proof of Corollary \ref{cor.superconvergence}]
A use of triangle inequality leads to
\begin{align}
\norm{\by_\mesh-\Pi_\disc \by_\disc}{} \le{}& \norm{\by_\mesh-\Pi_\disc y_\disc(\bu)}{}
+\norm{\Pi_\disc y_\disc(\bu)-\Pi_\disc y_\disc(\hat{u})}{} \nonumber\\ 
&+\norm{\Pi_\disc y_\disc(\hat{u})-\Pi_\disc \by_\disc }{}. 
\label{y.tri.ineq}
\end{align}
Consider the first term on the right hand side of \eqref{y.tri.ineq}. Using the Assumption  $\assum{1}$-i), we obtain
\be \label{y.first.term}
\norm{\by_\mesh-\Pi_\disc y_\disc(\bu) }{}  \lesssim  h^2\norm{\bu+f}{H^1(\O)}.
\ee

Under the assumptions of Theorem \ref{thm.superconvergence}, the second term on the right hand side of \eqref{y.tri.ineq} is estimated by \eqref{yd.ineq},
and the third term is estimated by using \eqref{a3.ineq} and \eqref{a3new1.ineq}.
We plug these estimates alongside \eqref{y.first.term} into \eqref{y.tri.ineq},
and use \eqref{est.bu.H1} to conclude the proof of \eqref{eq_supercv.y}. The result for the adjoint variable can be derived similarly. 

The full $h^2$ estimates are obtained, under the assumptions of Theorem \ref{thm.fullsuperconvergence},
by following the same reasoning and using the improved estimate \eqref{improved.A2} on $A_{21}$ (which leads
to improved estimates \eqref{yd.ineq} and \eqref{a3new1.ineq}).



\end{proof}

\section{The case of Neumann BC, with distributed and boundary control} \label{section5}

\subsection{Model}

Consider the distributed and boundary
optimal control problem governed by elliptic equations with Neumann BC
given by:
\begin{subequations}\label{model.neumann}
	\begin{align}
	&  { \min_{(u,u_b) \in \Uad} J(y, u,u_b) \quad \textrm{ subject to } }  \label{cost1}\\
	&  {-\div(A\nabla y) +c_0y = f+ u  \quad \mbox{ in } \Omega,  }\label{state11} \\
	&  {A\nabla y\cdot\bfn_{\O}  =f_b+u_b \quad\mbox{ on $\partial\O$},} \label{state21}
	\end{align}
\end{subequations} 
where  $\Omega$, $A$ and $f$ are as in Section \ref{sec:model},
$f_b\in L^2(\dr\O)$, $c_0 >0 $ is a positive constant, $\bfn_\O$ is the outer unit normal to $\O$,
$y$ is the state variable, and $u,u_b$ are the control variables.
The cost functional is
\[
J(y, u,u_b) :=\frac{1}{2}\norm{y- \yd}{}^2 + \frac{\alpha}{2} \norm{u}{}^2
+\frac{\beta}{2} \norm{u_b}{L^2(\dr\O)}^2
\]
with $\alpha>0$ and $\beta>0$ being fixed regularization parameters
and $\yd\in L^2(\O)$ being the desired state variable.
The set of admissible controls $\Uad \subset L^2(\Omega)
\times L^2(\dr\O)$ is a non-empty, convex and closed set. 
For a general element $V\in L^2(\O)\times L^2(\dr\O)$, $v$ and $v_b$ denote its components
in $L^2(\O)$ and $L^2(\dr\O)$, that is, $V=(v,v_b)$.

It is well known that given $U=(u,u_b) \in \Uad$, there exists a unique weak solution $y(U) \in H^1(\Omega)$ of \eqref{state11}-\eqref{state21}. 
That is, $y(U) \in H^1(\Omega)$ such that, for all $w\in H^1(\O)$,
\begin{equation}\label{weak_state1}
a(y(U),w)=\int_\O (f+u) {w} \d\x + \int_{\dr\O} (f_b+u_b)\gamma(w)\d s(\x),
\end{equation}
where $a(z,w)=
\int_\O (A\nabla z\cdot\nabla w +c_0zw)\d\x$ and $\gamma:H^1(\O) \rightarrow L^2(\dr\O)$ is
the trace operator. 

\medskip

Here and throughout, $\norm{\cdot}{\dr}$ and $(\cdot,\cdot)_\dr$ denote the norm and scalar product in $L^2(\dr\O)$. We also denote $\SCAL{\cdot}{\cdot}$ as the scalar
product on $L^2(\O)\times L^2(\dr\O)$ defined by
\[
\forall U,V\in L^2(\O)\times L^2(\dr\O)\,,\quad
\SCAL{U}{V}=\alpha(u,v)+\beta(u_b,v_b)_\dr.
\]
The convex control problem \eqref{model.neumann} has a unique solution $(\by,\bU) \in H^1(\O) \times \Uad $ and there exists a co-state $\bp \in H^1(\O)$ such that the triplet $(\by, \bp, \bU) \in H^1(\O) \times H^1(\O) \times \Uad$ satisfies the Karush-Kuhn-Tucker (KKT) optimality conditions \cite{jl}:
\begin{subequations} \label{continuous1}
	\begin{align}
	& a(\by,w) = (f + \bu, w) + (f_b+\bub,\gamma(w))_\dr\,  &\   \forall \: w \in H^1(\O), \\
	& a(w,\bp) = (\by-\yd, w)\,  &\   \forall \: w \in H^1(\O), \label{adj_cont1}\\
	& \SCAL{\bU+\bP_{\alpha,\beta}}{V-\bU}\geq     0\,  &\  \forall \: V\in  \Uad, \label{opt_cont1}
	\end{align}
\end{subequations}
where $\bP_{\alpha,\beta}=(\alpha^{-1} \bp,\beta^{-1}\gamma(\bp))$.

\subsection{The GDM for elliptic equations with Neumann BC} \label{section6}

\begin{definition}[GD for Neumann BC with reaction]\label{def:GD1}
	A gradient discretisation for Neumann BC is a quadruplet $\disc=(X_{\disc},\Pi_\disc,\trrec_\disc,\nabla_\disc)$ such that
	\begin{itemize}
		\item $X_{\disc}$ is a finite dimensional space of degrees of freedom,
		\item $\Pi_\disc:X_{\disc}  \rightarrow L^2(\O)$ is a linear mapping that reconstructs a function from
		the degrees of freedom,
		\item $\trrec_\disc:X_{\disc}  \rightarrow L^2(\dr\O)$ is a linear mapping that reconstructs a trace
		from the degrees of freedom,
		\item $\nabla_\disc:X_{\disc}  \rightarrow L^2(\O)^n$ is a linear mapping that reconstructs
		a gradient from the degrees of freedom.
		\item The following quantity is a norm on $X_{\disc}$:
		\be\label{def:norm1}
		\norm{w}{\disc}:= \norm{\nabla_\disc w}{}+\norm{\Pi_\disc w}{}.
		\ee
	\end{itemize}
\end{definition}

If $F\in L^2(\O)$ and $G\in L^2(\dr\O)$,
a GS for a linear elliptic problem
\be\label{base1}
\left\{
\ba
-\div(A\nabla \psi)+c_0 \psi=F\mbox{ in $\O$},\\
A\nabla \psi\cdot\bfn_\O=G\mbox{ on $\dr\O$}
\ea
\right.
\ee
is then obtained from a GD $\disc$ by writing:
\be\label{base.GS1}
\begin{aligned}
	&\mbox{Find $\psi_\disc\in X_{\disc}$ such that, for all $w_\disc\in X_{\disc}$,}\\
	&a_\disc(\psi_\disc,w_\disc)=
	\int_\O F\Pi_\disc w_\disc\d\x+\int_{\dr\O} G\trrec_\disc w_\disc\d s(\x),
\end{aligned}
\ee
where $a_\disc(\psi_\disc,w_\disc)=\int_\O (A\nabla_\disc \psi_\disc\cdot\nabla_\disc w_\disc +c_0\Pi_\disc \psi_\disc \Pi_\disc w_\disc)\d\x.$

For Neumann boundary value problems, the quantities $C_\disc$, $S_\disc$ and $W_\disc$
measuring the accuracy of the GS are defined as follows.
\be\label{def.CD1}
C_\disc := \max_{w\in X_{\disc}\backslash\{0\}} 
\left(\frac{\norm{\trrec_\disc w}{\dr}}{\norm{w}{\disc}}, \frac{\norm{\Pi_\disc w}{}}{\norm{w}{\disc}}\right).
\ee
\begin{align}
	\forall{}&\varphi\in H^1(\O)\,, \nonumber\\
	&S_\disc(\varphi)=\min_{w\in X_{\disc}}\Big(\norm{\Pi_\disc w-\varphi}{}
	+\norm{\trrec_\disc w-\gamma(\varphi)}{\dr}	+\norm{\nabla_\disc w-\nabla\varphi}{}\Big)\label{def.SD1}.\\
	\forall{}& \bvarphi\in H_{\div,\dr}(\O)\,,\nonumber \\
	&W_\disc(\bvarphi)=\max_{w\in X_{\disc}\backslash\{0\}}
\frac{1}{\norm{w}{\disc}}\Bigg|\int_\O\!\! \Pi_\disc w \div(\bvarphi)
		+\nabla_\disc w\cdot\bvarphi\d\x \nonumber \\
	& \qquad \qquad \qquad \qquad \qquad \qquad \qquad \qquad - \int_{\dr\O}\trrec_\disc w\gamma_{\bfn}(\bvarphi)\d s(\x)\Bigg|,\label{def.WD1}
\end{align}
where $\gamma_{\bfn}$ is the normal trace on $\dr\O$, and
$H_{\div,\dr}(\O)=\{\bvarphi\in L^2(\O)^n\,:\,\div(\bvarphi)\in L^2(\O)\,,\;
\gamma_{\bfn}(\bvarphi)\in L^2(\dr\O)\}$.

Using these quantities, we define $\WS_\disc$ as in \eqref{def.ws} and we have
the following error estimate.

\begin{theorem}[Error estimate for the PDE with Neumann BC]
\label{th:error.est.PDE1} Let $\disc$ be a GD in the
	sense of Definition \ref{def:GD1}, let $\psi$ be the solution in $H^1(\O)$
	to \eqref{base1}, and let $\psi_\disc$  be the solution to \eqref{base.GS1}. Then
\[
	\norm{\Pi_\disc \psi_\disc-\psi}{}+\norm{\nabla_\disc \psi_\disc-\nabla\psi}{}
	+\norm{\trrec_\disc \psi_\disc-\gamma(\psi)}{\dr}
	\lesssim \WS_\disc(\psi).
\]
\end{theorem}

\begin{proof}
	The estimate 
	\be\label{est.PiD.nablaD1}
	\norm{\Pi_\disc \psi_\disc - \psi}{}+\norm{\nabla_\disc \psi_\disc-\nabla \psi}{}
	\lesssim \WS_\disc(\psi)
	\ee
	is standard, and can be established as for homogeneous Dirichlet BC (see, e.g., \cite[Theorem 3.11]{koala}
	for the pure Neumann problem). The estimate on the traces is less standard,
	and hence we detail it now.
	Introduce an interpolant
	\[
	\mathcal P_\disc \psi\in \mathop{\rm argmin}_{w\in X_\disc}\Big(\norm{\Pi_\disc w-\psi}{}
	+\norm{\trrec_\disc w-\gamma(\psi)}{\dr}
	+\norm{\nabla_\disc w-\nabla\psi}{}\Big)
	\]
	and notice that
	\be\label{app.tr:interp.neu1}
	\norm{\Pi_\disc \mathcal P_\disc\psi-\psi}{}+
	\norm{\trrec_\disc \mathcal P_\disc\psi-\gamma(\psi)}{\dr}
	+\norm{\nabla_\disc \mathcal P_\disc\psi-\nabla\psi}{}
	\le {S}_\disc (\psi).
	\ee
	By definition of $C_\disc$ and of the norm $\norm{\cdot}{\disc}$,
	 for all $v\in X_\disc$,
	\[
	\norm{\trrec_\disc v}{\dr}\le C_\disc \left(\norm{\Pi_\disc v}{}+
	\norm{\nabla_\disc v}{}\right).
	\]
	Substituting $v=\psi_\disc-\mathcal P_\disc \psi$,
	a triangle inequality and \eqref{app.tr:interp.neu1} therefore lead to
	\begin{align}
	\Vert \trrec_\disc \psi_\disc&-\gamma(\psi)\Vert_{\dr}\nonumber\\
	\le{}&
	\norm{\trrec_\disc (\psi_\disc- \mathcal P_\disc\psi)}{\dr}
	+\norm{\trrec_\disc \mathcal P_\disc\psi-\gamma(\psi)}{\dr}\nonumber\\
	\le{}& C_\disc\left(\norm{\Pi_\disc \psi_\disc
		-\Pi_\disc \mathcal P_\disc\psi}{}+\norm{\nabla_\disc \psi_\disc-
		\nabla_\disc \mathcal P_\disc\psi}{}\right)
	+S_\disc(\psi).
	\label{end.est.tr1}
	\end{align}
	We then use the triangle inequality again and the estimates
	\eqref{est.PiD.nablaD1} and \eqref{app.tr:interp.neu1} to write
	\begin{align*}
	\Vert \Pi_\disc \psi_\disc-&\Pi_\disc \mathcal P_\disc\psi\Vert
	+\norm{\nabla_\disc \psi_\disc-
		\nabla_\disc\mathcal  P_\disc\psi}{}\\
	\le{}& 
	\norm{\Pi_\disc \psi_\disc-\psi}{}
	+\norm{\psi-\Pi_\disc \mathcal P_\disc\psi}{}
	+\norm{\nabla_\disc \psi_\disc-\nabla\psi}{}
	+\norm{\nabla\psi-\nabla_\disc \mathcal P_\disc\psi}{}\\
		\lesssim{}& \WS_\disc(\psi).
	\end{align*}
	The proof is complete by plugging this result in \eqref{end.est.tr1}.
\end{proof}




\subsection{The gradient discretisation method for the Neumann control problem}

Let $\disc$ be a GD as in Definition \ref{def:GD1}, $\Uh$ be a finite dimensional space of $L^2(\Omega)$, and set $\Uadh = \Uad \cap \Uh$.
A GS for \eqref{continuous1} consists in 
seeking $(\by_{\disc}, \bp_{\disc}, \bU_{h})\in X_{\disc} \times X_{\disc} \times \Uadh$,
with $\bU_h=(\bu_h,\bubh)$, such that
\begin{subequations} \label{discrete_kkt1}
	\begin{align}
	& a_{\disc}(\by_{\disc},w_{\disc}) = (f + \bu_{h}, \Pi_\disc w_{\disc})+(f_b+\bubh,\trrec_\disc w_\disc)_\dr \,
	&\forall \: w_{\disc} \in  X_{\disc},  \label{discrete_state1} \\
	& a_{\disc}(w_{\disc},\bp_{\disc}) = (\Pi_\disc \by_{\disc}-\yd, \Pi_\disc w_{\disc})  \,&  \forall \: w_{\disc} \in  X_{\disc}, \label{discrete_adjoint1} \\
	& \SCAL{\bU_h + \bP_{\disc,\alpha,\beta} }{V_h-\bU_h} \geq     0\,
	&\forall \: V_h \in  \Uadh, \label{opt_discrete1}
	\end{align}
\end{subequations} 
where $\bP_{\disc,\alpha,\beta}=
(\alpha^{-1}\Pi_\disc \bp_{\disc},\beta^{-1}\trrec_\disc {\bar p}_\disc)$.


Let $\Prh:L^2(\O)\times L^2(\dr\O)  \rightarrow \Uh$ be the $L^2$ orthogonal projection on $\Uh$ for the scalar product $\SCAL{\cdot}{\cdot}$. 
We denote the norm on $L^2(\O)\times L^2(\dr\O)$ associated to $\SCAL{\cdot}{\cdot}$
by $\NORM{\cdot}$, so that $\NORM{V}=\sqrt{\alpha\norm{v}{}^2+\beta\norm{v_b}{}^2}$.
If $W\in L^2(\O)\times L^2(\partial\O)$, we define
\[
E_h(W)=\NORM{W-\Prh W}.
\]

\begin{theorem}[Control estimate] \label{theorem.control.NB}
	Let $\disc$ be a GD in the sense of Definition \ref{def:GD1}, $\bU$ be the optimal control for \eqref{continuous1}
	and $\bU_h$ be the optimal control for the GS \eqref{discrete_kkt1}.
	We assume that
	\be\label{stab.proj1}
	\Prh(\Uad)\subset \Uadh.
	\ee
	Then there exists $C$ only depending on $\O$, $A$, $\alpha$, $\beta$ and an upper bound of $C_\disc$ such that
\[
	\NORM{\bU-\bU_h}\le C\Big(E_h(\bP_{\alpha,\beta})+E_h(\bU)
	+\WS_\disc(\bp)+\WS_\disc(\by)\Big).
\]
\end{theorem}
\begin{proof}
	The proof is identical to the proof of Theorem \ref{theorem.control.DB} (taking $\ud=0$),
with obvious substitutions (e.g. $\bP_{\disc,\alpha}\leadsto \bP_{\disc,\alpha,\beta}$
and $\bu_h\leadsto \bU_h$)
and the $L^2$ inner products $(\cdot,\cdot)$ replaced by $\SCAL{\cdot}{\cdot}$ whenever they
involve $\bP_{\disc,\alpha}$ or $\bu_h$.
\end{proof}

\begin{remark}[Super-convergence of the control for Neumann problems]
Using the same technique as in the proof of Theorem \ref{thm.superconvergence}, and
extending the assumptions \assum{1}--\assum{4} to boundary terms in a natural way
(based on trace inequalities and Sobolev embedding of $H^{1/2}(\partial\O)$),
an $\mathcal O(h^{3/2})$ super-convergence result can be proved on post-processed
controls for Neumann BC.
\end{remark}

\begin{remark}
Consider the distributed
optimal control problem governed by elliptic equations with Neumann BC
given by:
\begin{subequations}\label{model.neumann1}
	\begin{align}
	&  { \min_{u \in \Uad} J(y, u) \quad \textrm{ subject to } }  \label{cost2}\\
	&  {-\div(A\nabla y) = u  \quad \mbox{ in } \Omega,  }\label{state22} \\
	&  {A\nabla y\cdot\bfn_{\O}  =0 \quad\mbox{ on $\partial\O$},\quad \int_{\O}y(\x)d\x=0,} \label{state32}
	\end{align}
\end{subequations} 
where  $\Omega$ and $A$ are as in Section \ref{sec:model}.
The cost functional is \eqref{costfunctional} with $\bu_d=0$ and $\yd\in L^2(\O)$ is such that $\int_\O \by_d(\x)d\x=0$.
Fixing $a<0<b$, the admissible set of controls is chosen as
$$\Uad=\left\{u\in L^2(\O) \,:\,a\le u\le b \mbox{ a.e. and }\int_\O u(\x)d\x=0\right\}.$$

For a given $u \in \Uad$, there exists a unique weak solution $y(u) \in H^1_\star(\Omega):=\{w\in H^1(\O)\,:\int_\O w(\x) \d\x=0\}$ of \eqref{state22}--\eqref{state32}.\\

The convex control problem \eqref{model.neumann1} has a unique solution $(\by,\bu) \in H^1_\star(\O) \times \Uad $ and there exists a co-state $\bp \in H^1_\star(\O)$ such that the triplet $(\by, \bp, \bu) \in H^1_\star(\O) \times H^1_\star(\O) \times \Uad$ satisfies the Karush-Kuhn-Tucker (KKT) optimality conditions \cite{jl}:
\begin{subequations} \label{continuous2}
	\begin{align}
	& a(\by,w) = (\bu, w) \,  &\   \forall \: w \in H^1(\O), \\
	& a(w,\bp) = (\by-\yd, w)\,  &\   \forall \: w \in H^1(\O), \label{adj_cont2}\\
	& (\bp+\alpha\bu,v-\bu)\geq 0\,  &\  \forall \: v\in  \Uad, \label{opt_cont2}
	\end{align}
where $a(z,w)=
\int_\O A\nabla z\cdot\nabla w\d\x$.
	
\end{subequations}
Then, $\bu$ can be characterized in terms of the projection formula given by
 $$\bu=P_{[a,b]}(-\alpha^{-1}\bp+c),$$ where $c$ is a constant chosen such that $\int_\O \bu(x)d\x=0$. 
The existence and uniqueness of this $c$ follows by noticing that $\Gamma(c):=\int_{\O}P_{[a,b]}(-\alpha^{-1}\bp+c)$ is continuous, $\lim_{c \to -\infty}\Gamma(c)=a|\O|<0$, $\lim_{c \to +\infty}\Gamma(c)=b|\O|>0$, and $\Gamma$ is strictly increasing around $c$ if $\Gamma(c)=0$. The adaptation of the theoretical analysis and numerical algorithms for this problem is a topic of future research.

\end{remark}

\section{Numerical results} \label{sec_examples}
In this section, we present numerical results to support the theoretical estimates obtained in the previous sections. We use three specific schemes for the state and adjoint variables: conforming finite element method, non-conforming finite element method, and hybrid mimetic mixed (HMM) method (a family that contains, the hMFD schemes analysed for example in \cite{bre-05-con}, owing to the results
in \cite{dro-10-uni}).
We refer to \cite{DEH15} for the description of the GDs corresponding to these methods (see also
Section \ref{sec:LinftyHMM}, Appendix for the HMM GD). The control variable is discretised using piecewise constant functions. The discrete solution is computed by using the primal-dual active set algorithm, see \cite[Section 2.12.4]{tf}.

 Let the relative errors be denoted by
\[
\err_\disc(\by):=\frac{\norm{\Pi_\disc \by_\disc -\by_\tau}{}}{\norm{\by_\tau}{}},\quad
\err_\disc(\nabla\by) :=\frac{\norm{\nabla_\disc \by_\disc -\nabla\by}{}}{\norm{\nabla\by}{}}
\]
\[
\err_\disc(\bp):=\frac{\norm{\Pi_\disc \bp_\disc -\bp_\tau}{}}{\norm{\bp_\tau}{}},\quad
\err_\disc(\nabla\bp) :=\frac{\norm{\nabla_\disc \bp_\disc -\nabla\bp}{}}{\norm{\nabla\bp}{}}
\]
\[
\err(\bu):=\frac{\norm{\bu_h -\bu}{}}{\norm{\bu}{}}
\quad\mbox{ and }\quad
\err(\tu) :=\frac{\norm{\tu_h - \tu}{}}{\norm{\bu}{}}.
\]
Here, the definitions of $\tu$ and $\tu_h$ follow from \eqref{Projection2} and
the discussion in Section \ref{sec:disc.pp}:
\begin{itemize}
\item For FE methods, 
\[
\tu=P_{[a,b]}(\Proj\ud-\alpha^{-1}\bp)\mbox{ and }\tu_h=P_{[a,b]}(\Proj\ud-\alpha^{-1}\Pi_\disc\bp_\disc).
\]
\item For HMM methods, 
\begin{align*}
&\tu_{|K}=P_{[a,b]}(\Proj\ud-\alpha^{-1}\bp(\overline{\x}_K))\mbox{ for all $K\in\mesh$, and }\\
&\tu_h=P_{[a,b]}(\Proj\ud-\alpha^{-1}\Pi_\disc\bp_\disc)=\bu_h.
\end{align*}
\end{itemize}

 The $L^2$ errors of state and adjoint variables corresponding to the FE methods are computed using a seven point Gaussian quadrature formula, and the energy norms are calculated using midpoint rule. In the case of HMM, both the energy and $L^2$ norms are computed using the midpoint rule. The $L^2$  errors of control variable is computed using a three point Gaussian quadrature formula. The post-processed control corresponding to the FE methods is evaluated using a seven point Gaussian quadrature formula, whereas for the HMM methods, the post-processed control is computed using midpoint rule. For HMM methods, we can use simpler quadrature rules owing to
the fact that the reconstructed functions are piecewise constants.
These errors are plotted against the mesh parameter $h$ in the log-log scale.

\subsection{Dirichlet BC}

The model problem is constructed in such a way that the exact solution is known.

\subsubsection{Example 1}
\label{example1}
This example is taken from \cite{ABV13}.
In this experiment, the computational domain  $\O$ is taken to be the unit square $\left(0,1\right)^2$. The data in the optimal distributed control problem \eqref{cost}--\eqref{state2} are chosen as follows:
\begin{align*}
&\by=\sin(\pi x)\sin(\pi y), \quad \bp=\sin(\pi x)\sin(\pi y),\\
&   \ud=1-\sin(\pi x/2)-\sin(\pi y/2), \quad \alpha=1,\\
&\Uad=[0,\infty),\quad \bu=\max(\ud - \bp,0).
\end{align*}
The source term $f$ and the desired state  $\yd$ are the computed using
\begin{equation*}
f=-\Delta \by-\bu, \quad \yd=\by+\Delta \bp.
\end{equation*}
Figure \ref{femmesh} shows the initial triangulation of a square domain and its uniform refinement.

\begin{figure}[ht]
	\centering
	\begin{minipage}[b]{0.4\linewidth}
		\includegraphics[width=6.1cm]{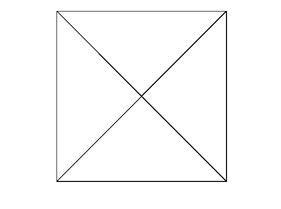}
	\end{minipage}
	\quad
	\begin{minipage}[b]{0.4\linewidth}
		\includegraphics[width=5.4cm]{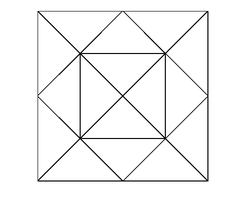}
	\end{minipage}
	\caption{Initial triangulation and its uniform refinement}
	\label{femmesh}
\end{figure}

Since $\O$ is convex, Theorems \ref{theorem.control.DB} and \ref{thm.superconvergence}
(see also the discussion before Section \ref{sec:disc.pp}), Proposition \ref{theorem.state.adj.DB} and Corollary \ref{cor.superconvergence} predict linear order of convergence for the state and adjoint variable in the energy norm, nearly quadratic order of convergence for state and adjoint variables in $L^{2}$ norm, linear order of convergence for the control variable in $L^{2}$ norm, and a nearly quadratic order of convergence for the post-processed control. These nearly-quadratic convergence properties only occur in case of a super-convergence result for the state equation (i.e. Estimate \eqref{state:scv}), which is always true for the FE methods but depends on some choice of points for the HMM scheme (see \cite{jd_nn}, and below).

\medskip 
\textbf{Conforming  FE method:}
The discrete solution is computed on several uniform grids with mesh sizes $h = \frac{1}{2^i}, i=2,\ldots,6$. The error estimates and the convergence rates of the control, the state and the adjoint variables are calculated. 
 The post-processed control is also computed. Figure \ref{fig2} displays the convergence history of the error on uniform meshes. As noticed in Remark \ref{conformingP1} and as already seen in \cite{CMAR}, we obtain linear order of convergence for the control and quadratic convergence for the post-processed control. The theoretical rates of convergence are confirmed by these numerical outputs.
%
%
%
				
				\begin{center}
					\begin{figure}[!h]
						\centering
						{\includegraphics[width=11.cm]{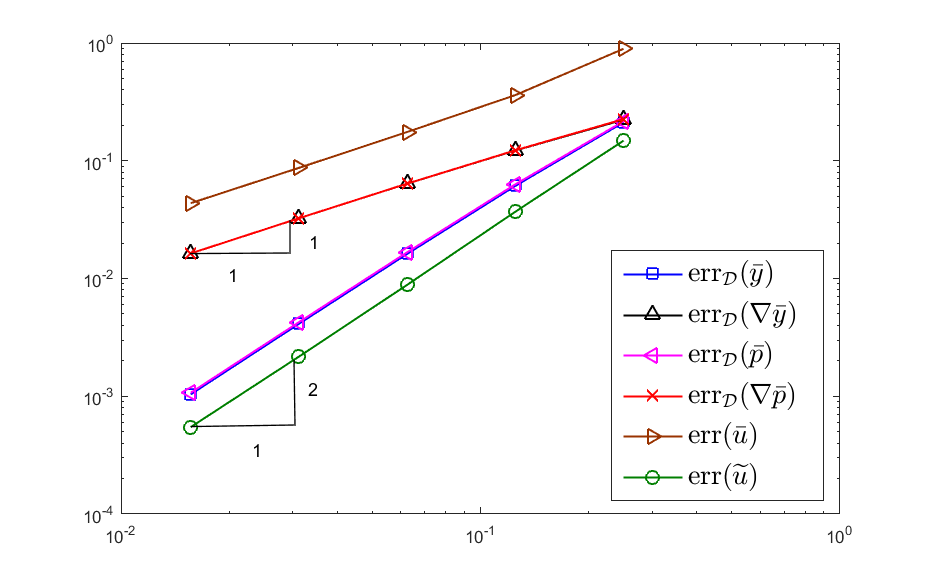}}
						\caption{Dirichlet BC, example 1, conforming FE method}
						\label{fig2}
					\end{figure}
					
				\end{center}
	\medskip	
%
%
\textbf{Non-Conforming  FE method: }
For comparison, we compute the solutions of the nc$\mathbb{P}_1$ finite element method on the same grids. The errors of the numerical approximations to state, adjoint and control variables on uniform meshes are evaluated. 
The convergence behaviour of state, adjoint and control variables is illustrated in Figure \ref{fig3}. 
Here also, these outputs confirm the theoretical rates of convergence.

%
%
%
%
%
%
%
%
		\begin{center}
			\begin{figure}[!h]
				\centering
				{\includegraphics[width=11.cm]{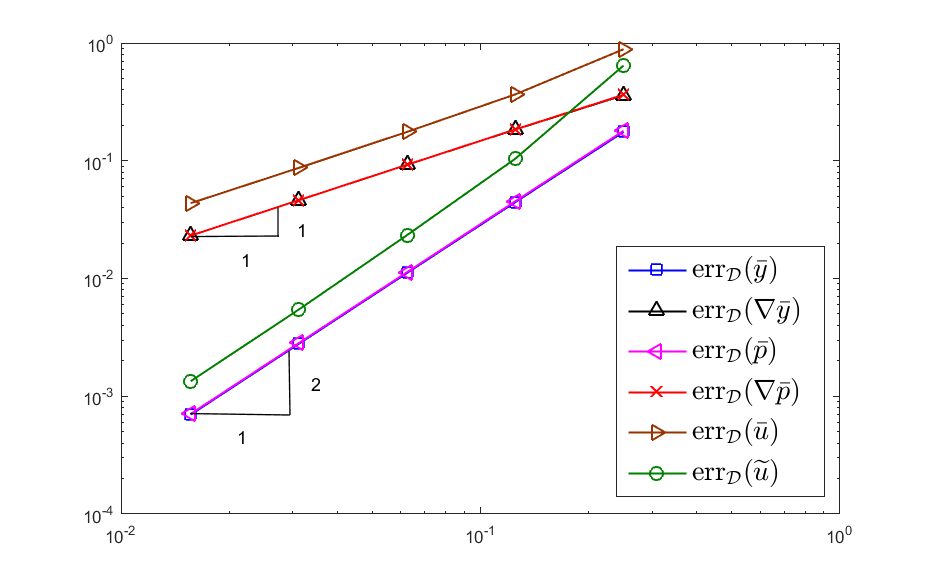}}
				\caption{Dirichlet BC, example 1, non-conforming FE method}
				\label{fig3}
			\end{figure}
			
		\end{center}

\textbf{HMM scheme:}
In this section, the schemes were first tested on a series of regular triangle meshes from \cite{benchmark}
(see Figure \ref{fig-HMMmeshes}, left) where the points $\mathcal{P}$ (see \cite[Definition 2.21]{DEH15}) are located at the center of gravity of the cells $\left({\bf{Test 1}}\right)$. For such meshes, the state and adjoint equations enjoy a super-convergence
property in $L^2$ norm \cite{bre-05-fam,jd_nn} and thus, as expected,
so does the scheme for the entire control problem after projection of the exact control.
 In Figure \ref{fig5}, the graph of the relative errors corresponding to control, state and adjoint variables against the discretisation parameter is plotted in the log−log scale. {\bf{Test 2}} focuses  on a cartesian grid where the points $\mathcal{P}$ are shifted away from the centre of gravity (see Figure \ref{fig-HMMmeshes}, right). For such a sequence of meshes, it has been
observed in \cite{jd_nn} that the HMM method can display a loss of superconvergence for the state equation. It is therefore expected that the same loss occurs, for all variables, for the control problem. This can be clearly seen in Figure \ref{fig6}.

\begin{figure}
\begin{tabular}{cc}
\resizebox{0.4\linewidth}{!}{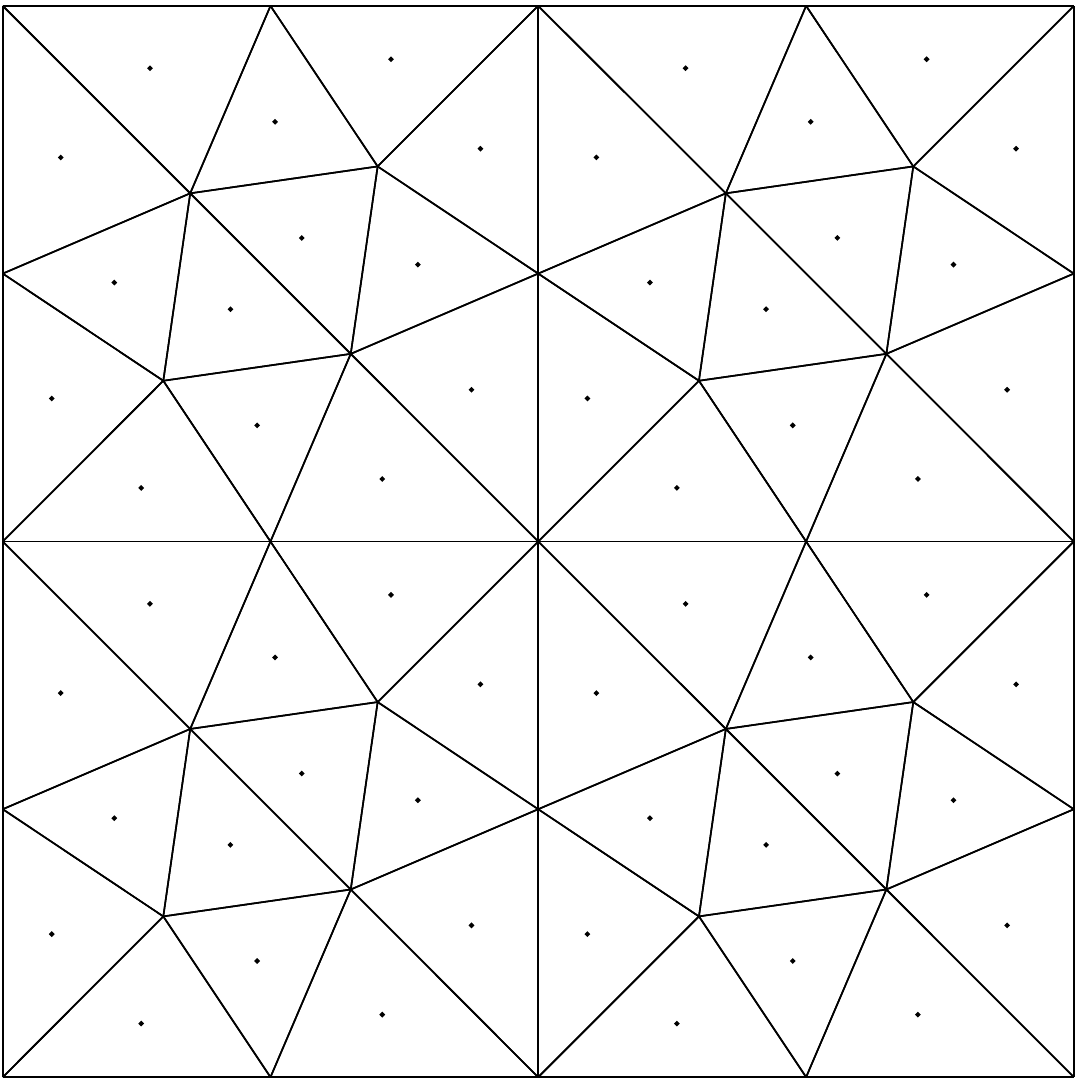}&
\resizebox{0.4\linewidth}{!}{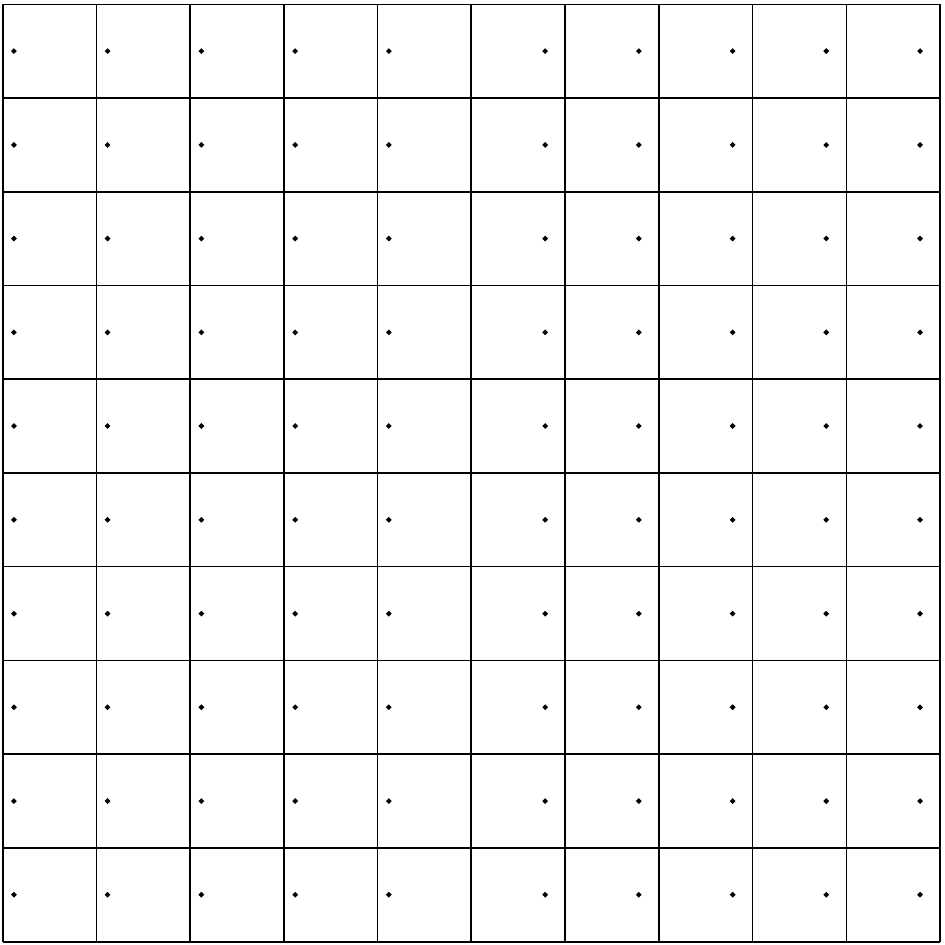}
\end{tabular}
\caption{Mesh patterns for the tests using the HMM method (left: Test 1;
right: Test 2).}
\label{fig-HMMmeshes}
\end{figure}

%
%
%

		\begin{center}
			\begin{figure}[!h]
				\centering
				{\includegraphics[width=11.cm]{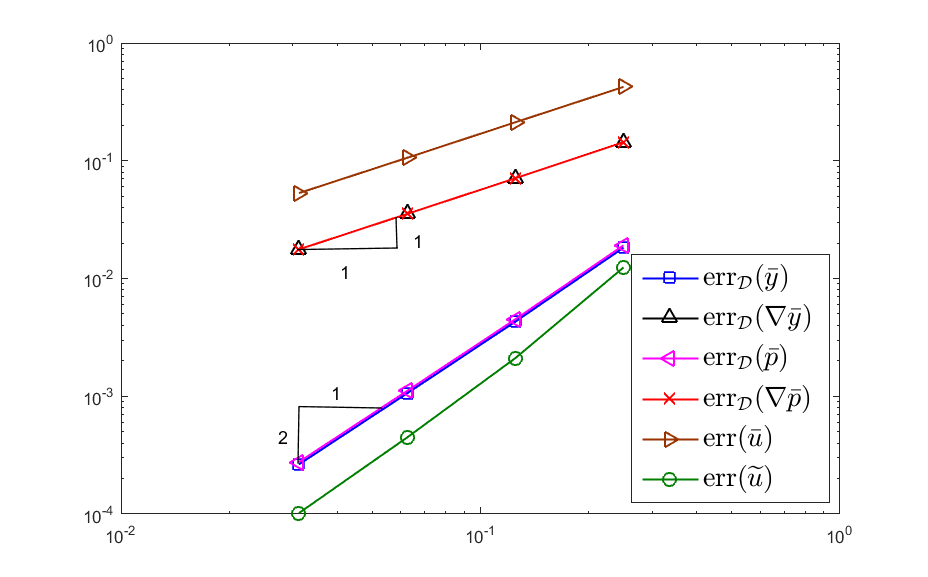}}
				\caption{Dirichlet BC, example 1, HMM $(\bf{Test 1})$}
				\label{fig5}
			\end{figure}
			
		\end{center}

%
%
%
%
%

		\begin{center}
			\begin{figure}[!h]
				\centering
				{\includegraphics[width=11.cm]{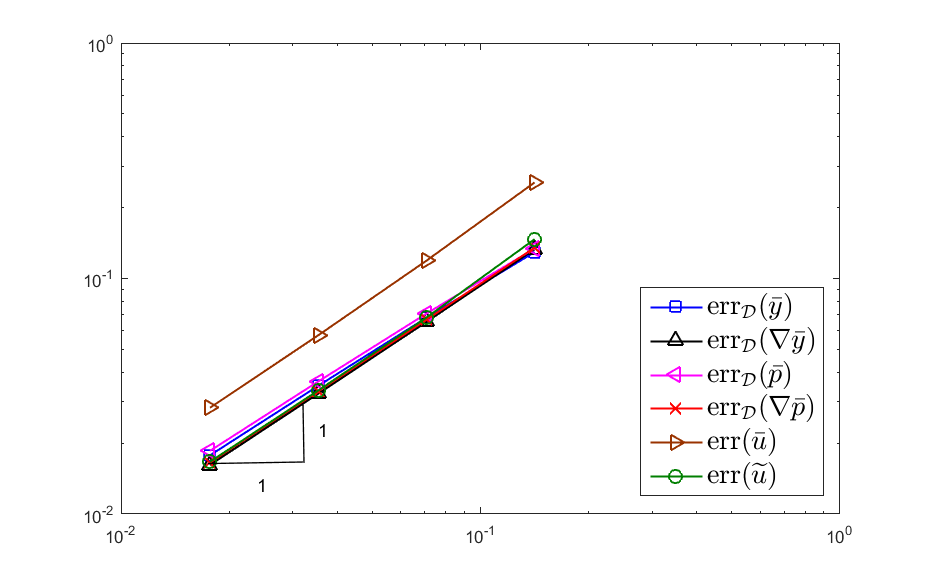}}
				\caption{Dirichlet BC, example 1,  HMM $(\bf{Test 2})$ }
				\label{fig6}
			\end{figure}
			
		\end{center}

\subsubsection{Example 2}
\label{example2}

In this example, we report the results of numerical tests
carried out for the L-shaped domain $\O= \left(-1,1\right)^2 \setminus \left( \left[0,1\right) \times \left(-1,0\right] \right)$.
 The exact solutions are chosen as follows, and correspond to $\ud=0$.
\begin{align*}
&\by(r,\theta)=\left( r^2\cos^2\theta-1\right) \left(r^2\sin^2\theta-1\right)r^{2/3}g\left(\theta\right),
\quad\Uad=[-600,-50],  \\
& \alpha=10^{-3}, \quad   \bu=  P_{[-600, -50]}\left( - \frac{1}{\alpha} \bp\right) 
\end{align*}
where $g\left(\theta\right)=\left(1-\cos\theta\right) \left(1+\sin\theta\right)$ and $\left(r,\theta\right)$ are the polar coordinates. 
The source term $f$ and the desired state  $\yd$ can be determined using the above functions.
The interest of this test-case is the loss of $H^2$-regularity property for the state
and adjoint equations.
Figure \ref{femLshapemesh} shows the initial triangulation of a L-shape domain and its uniform refinement.
\begin{figure}[ht]
\centering
\begin{minipage}[b]{0.45\linewidth}
	\includegraphics[width=6.cm]{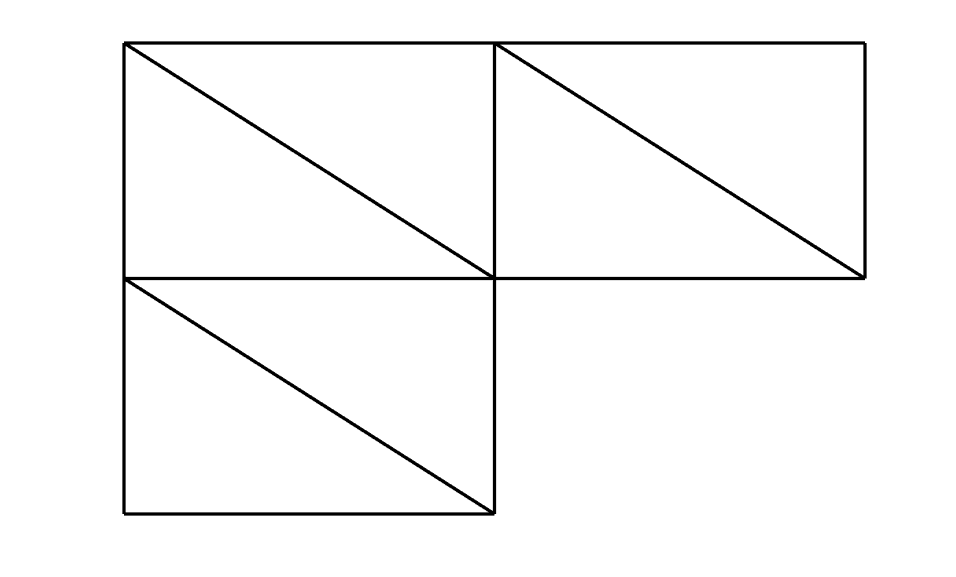}
\end{minipage}
\quad
\begin{minipage}[b]{0.45\linewidth}
	\includegraphics[width=6.cm]{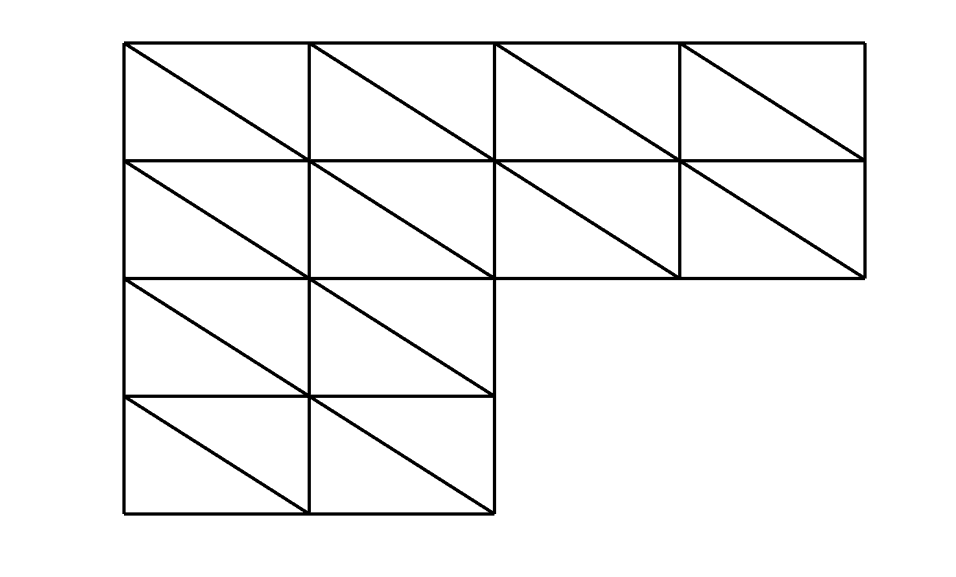}
\end{minipage}
\caption{Initial triangulation and its uniform refinement}
\label{femLshapemesh}
\end{figure}

\textbf{Conforming  FE method:}
The errors in the energy norm and the $L^2$ norm together with their orders of
convergence are evaluated. 
These numerical order of convergence clearly match the expected order of convergence, given the regularity property of the
exact solutions. The convergence rates are plotted in the log-log scale in Figure \ref{fig8}.
%
%
%
%
%
%
%
%
%
		\begin{center}
			\begin{figure}[!h]
				\centering
				{\includegraphics[width=11.cm]{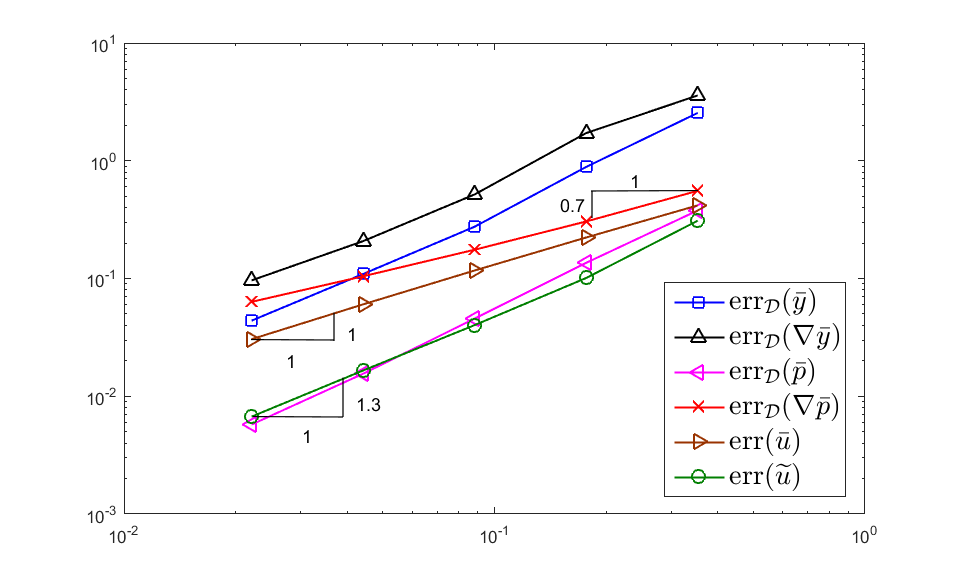}}
				\caption{Dirichlet BC, example 2 (L-shaped domain), conforming FE method}
				\label{fig8}
			\end{figure}
			
		\end{center}

\textbf{Non-Conforming  FE  method:}
		The errors between the true and computed solutions are computed 
		for different mesh sizes. In Figure \ref{fig9}, we plot the $L^2$-norm and $H^1$ norm of the error against the mesh parameter $h$. 
	
%
%
%
%

			\begin{center}
				\begin{figure}[!h]
					\centering
					{\includegraphics[width=11.cm]{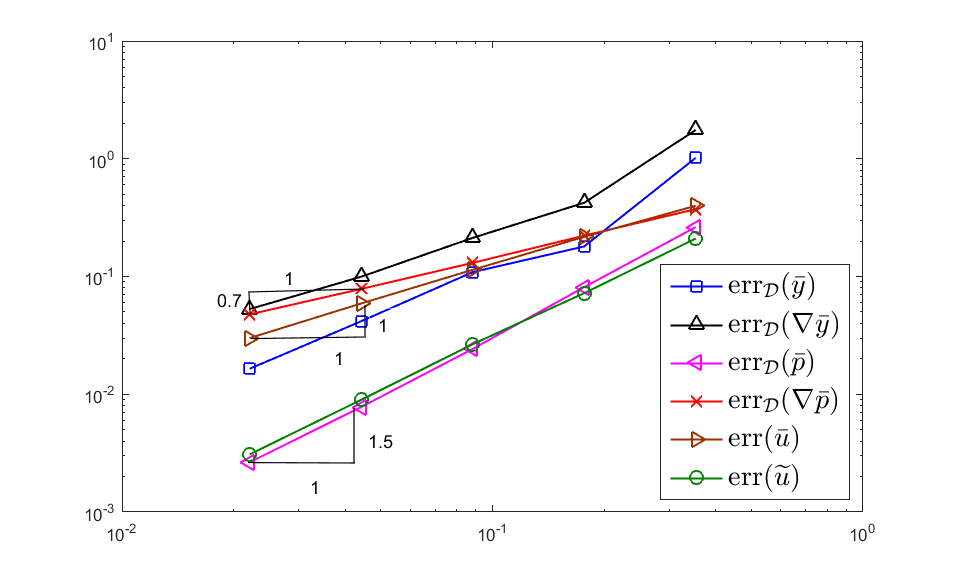}}
					\caption{Dirichlet BC, example 2 (L-shaped domain), non-conforming FE method}
					\label{fig9}
				\end{figure}
				
			\end{center}		
		
\textbf{HMM method:}
The errors corresponding to control, adjoint and state variables are computed using HMM $(\bf{Test \; 1})$. In Figure \ref{fig10}, the graph of the errors are plotted against the mesh size $h$ in the log-log scale.

%
%
%
%

		\begin{center}
			\begin{figure}[h!]
				\centering
				{\includegraphics[width=11.cm]{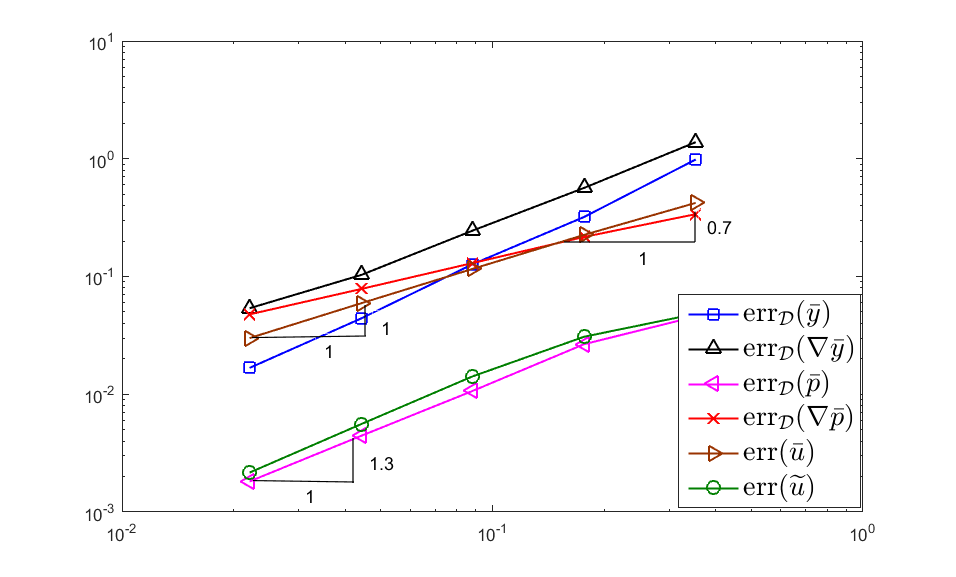}}
				\caption{Dirichlet BC, example 2 (L-shaped domain), HMM}
				\label{fig10}	
					\end{figure}
		\end{center}

Since $\O$ is non-convex, we obtain only suboptimal orders of convergence for the state and adjoint variables in the energy norms and $L^2$ norms. Also we observe suboptimal order of convergence for the post processed control. However, the control converges at the optimal rate of $h$.

\subsection{Neumann BC}

In this example, we consider the optimal control problem  defined by \eqref{model.neumann} with  $\O=\left(0,1\right)^2$ and $c_0=1$. We choose the exact state variable $\by$ and the adjoint variable $\bp$ as
\begin{equation}\label{example3}
\begin{aligned}
&\by=\frac{-1}{\pi}(\cos(\pi x)+\cos(\pi y)),  \quad \bp=\frac{-1}{\pi}(\cos(\pi x)+\cos(\pi y)),\\
&\Uad=[-750,-50],\quad \alpha=10^{-3}, \quad \bu(\x)= P_{[-750, -50]}\left( - \frac{1}{\alpha} \bp(\x) \right).
\end{aligned}
\end{equation}
We therefore have $\ud=0$. The source term $f$ and the observation $\yd$ can be computed using
\begin{equation*}
f=-\Delta \by+\by-\bu, \quad \yd=\by+\Delta \bp-\bp.
\end{equation*}
\textbf{Conforming  FE method:}
The errors and the orders of convergence for the control, state and adjoint variables are calculated for different mesh parameter $h$. The numerical errors are plotted against the discretisation parameter in the log-log scale in Figure \ref{fig11}.

%
%
%
%

		\begin{center}
			\begin{figure}[!h]
				\centering
				{\includegraphics[width=11.cm]{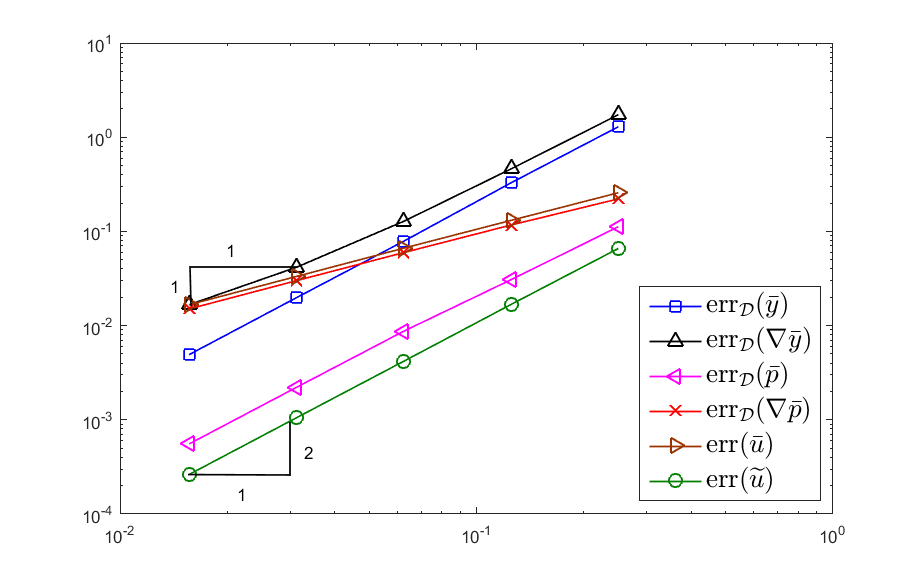}}
				\caption{Neumann BC, test case corresponding to \eqref{example3},
 conforming FE method}
				\label{fig11}
			\end{figure}
		\end{center}

\textbf{Non-Conforming  FE  method:}
The error estimates and the convergence rates of the control, the state and the adjoint variables are evaluated. 
The post-processed control is also computed. Figure \ref{fig12} displays the convergence history of the error on uniform meshes.

%
%

	\begin{center}
		\begin{figure}[!h]
			\centering
			{\includegraphics[width=11.cm]{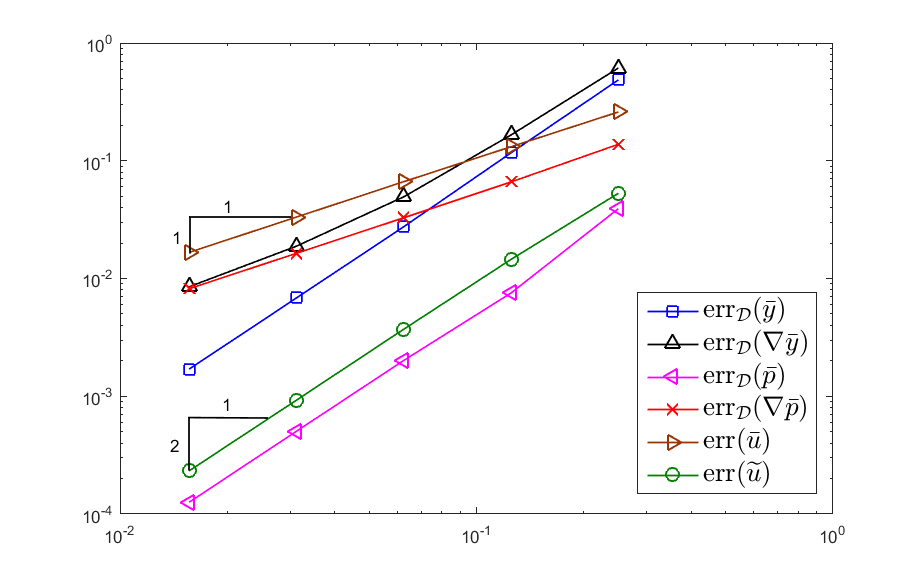}}
			\caption{Neumann BC, test case corresponding to \eqref{example3}, non-conforming FE method}
			\label{fig12}
		\end{figure}
	\end{center}

The observed orders of convergences agree with the predicted ones as seen in the figures.

\section{Appendix}

\subsection{$L^\infty$ estimates for the HMM method}\label{sec:LinftyHMM}

Let us first briefly recall the GD corresponding to the HMM method 
\cite{koala,dro-12-gra}. We consider a polytopal mesh $\polyd=(\mesh,\edges,\centers)$
as defined in \cite[Definition 7.2]{koala}: $\mesh$ is the set of cells (generic notation $K$),
$\edges$ is the set of faces (generic notation $\sigma$) and $\centers$ is a set
made of one point per cell (notation $\x_K$ -- this point does not need to be the center
of mass of $K$ in general). If $K\in\mesh$ then $\edges_K$ is the set of faces of
$K$. For $\edge\in\edgescv$, $|\edge|$ is the measure of $\edge$,
$\overline{\x}_\edge$ is the center of mass of $\edge$,
$d_{K,\edge}=(\overline{\x}_\edge-\x_K)\cdot\bfn_{K,\edge}$ is the orthogonal distance between
$\x_K$ and $\edge$, $\bfn_{K, \sigma}$ is the outer normal to $K$ on
$\edge$, and $D_{K,\edge}$ is the convex hull of $\x_K$ and $\edge$.
An HMM GD is defined the following way.

\begin{itemize}
\item The degrees of freedom are made of one value in each cell and one value on each
edge, so $X_{\disc,0}=\{v=((v_K)_{K\in\mesh},(v_\edge)_{\edge\in\edgescv})\,:\,
v_K\in\R\,,\;v_\edge\in\R\,,\;v_\edge=0\mbox{ if $\edge\subset \partial\Omega$}\}$.
\item The reconstructed functions are piecewise constant in the cells: for $v\in X_{\disc,0}$,
$\Pi_\disc v\in L^2(\O)$ is defined by $(\Pi_\disc v)_{|K}=v_K$ for all $K\in\mesh$.
\item The reconstructed gradient is piecewise constant in the sets $(D_{K,\edge})_{K\in\mesh,\,
\edge\in\edgescv}$: if $v\in X_{\disc,0}$, then $\nabla_\disc v\in L^2(\O)^n$ is defined by
\begin{align*}
&\forall K\in\mesh\,,\;\forall \edge\in\edgescv\,,\\
&\quad (\nabla_\disc v)_{|D_{K,\edge}}=\overline{\nabla}_K v + \frac{\sqrt{n}}{d_{K,\edge}}(v_\edge-v_K-\overline{\nabla}_K v
\cdot(\overline{\x}_\edge-\x_K))\bfn_{K,\edge},
\end{align*}
where
\[
\overline{\nabla}_K v = \frac{1}{|K|}\sum_{\edge\in\edgescv}|\edge|v_\edge\bfn_{K,\edge}
\]
(this gradient is perhaps the most natural choice; though not the only possible choice
within the HMM family; see \cite{koala,dro-12-gra} for a more complete presentation).
\end{itemize}

Under standard local regularity assumptions on the mesh,
\cite[Proposition 12.14 and 12.15]{koala} yield
the following error estimate on $\WS_\disc$: if $A$ is Lipschitz-continuous and
$\psi\in H^2(\O)$, for some $C$ not depending on $\psi$ or $\polyd$:
\be\label{est.WS}
\WS_\disc(\psi)\le Ch\norm{\psi}{H^2(\O)}.
\ee

\medskip

Under the quasi-uniformity assumption on the mesh, we establish the following
$L^\infty$ error estimate and bound for the HMM.

\begin{theorem}[$L^\infty$ estimates for HMM]\label{th:Linfty}
Consider the dimension $n=2$ or $3$. Let $\polyd$ be a polytopal mesh and $\disc$ be an HMM gradient discretisation.
Take $\varrho\ge  \theta_\polyd+\zeta_\disc+\chi_\polyd$, where $\theta_{\polyd}$ and $\zeta_\disc$
are defined by \cite[Eqs. (7.8) and (12.18)]{koala}, and
\[
\chi_\polyd=\max_{K\in\mesh}\frac{h^n}{|K|}.
\]
Assume that $A$ is Lipschitz-continuous, that $\Omega$ is convex and that $F\in L^2(\O)$. 
There exists then $C$, depending only on $\O$, $A$ and $\varrho$, such that,
if $\psi$ solves \eqref{base} and $\psi_\disc$ solves \eqref{base.GS},
\be\label{Linfty.error}
\norm{\Pi_\disc\psi_\disc -\psi_\mesh}{L^\infty(\O)}\le C \norm{F}{}\left\{\begin{array}{ll}h|\ln(h)|&\mbox{ if $n=2$},\\
h^{1/2}&\mbox{ if $n=3$},\end{array}\right.
\ee
where $(\psi_\mesh)_{|K}=\psi(\x_K)$ for all $K\in\mesh$, and
\be\label{Linfty.est.HMM}
\norm{\Pi_\disc\psi_\disc}{L^\infty(\O)}\le C \norm{F}{}.
\ee
\end{theorem}

\begin{proof} In this proof, $X\lesssim Y$ means that $X\le MY$ for some $M$
depending only on $\O$, $A$ and $\varrho$.
The theorem's assumptions ensure that $\psi\in H^2(\O)\cap H^1_0(\Omega)\subset C(\overline{\O})$. Let
$v=((\psi(\x_K))_{K\in\mesh},(\psi(\overline{\x}_\edge))_{\edge\in\edges})\in
X_{\disc,0}$. By the proof of \cite[Proposition A.6]{koala} (see also  \cite[(A.10)]{koala}),
\[
\norm{\Pi_\disc v-\psi}{}+\norm{\nabla_\disc v-\nabla\psi}{}
\lesssim h \norm{\psi}{H^2(\O)}\lesssim h\norm{F}{}.
\]
A use of \eqref{est.WS} and the triangle inequality then gives
\be\label{Linfty.1}
\norm{\nabla_\disc (v-\psi_\disc)}{}\lesssim h \norm{F}{}.
\ee
\cite[Lemma B.12]{koala} establishes the following discrete Sobolev embedding,
for all $q\in [1,6]$ if $n=3$ and all $q\in [1,+\infty)$
if $n=2$:
\be\label{disc.sob}
\forall w\in X_{\disc,0}\,,\;\norm{\Pi_\disc w}{L^q(\O)}\lesssim q \norm{\nabla_\disc w}{}.
\ee
An inspection of the constants appearing
in the proof of \cite[Lemma B.12]{koala} shows that the inequality $\lesssim$ in \eqref{disc.sob}
is independent of $q$. Substitute $w=v-\psi_\disc$ in \eqref{disc.sob} and use   \eqref{Linfty.1} to obtain
\[
\norm{\Pi_\disc (v-\psi_\disc)}{L^q(\O)}\lesssim h q \norm{F}{}.
\]
Let $K_0\in\mesh$ be such that $\norm{\Pi_\disc(v-\psi_\disc)}{L^\infty(\O)}=|\Pi_\disc(v-\psi_\disc)_{|K_0}|$ and
write
\begin{align*}
\norm{\Pi_\disc(v-\psi_\disc)}{L^\infty(\O)}={}&|K_0|^{-1/q}\left(|K_0|\,|\Pi_\disc(v-\psi_\disc)_{|K_0}|^q\right)^{1/q}\\
\le{}& \chi_\polyd^{1/q}h^{-\frac{n}{q}} \norm{\Pi_\disc (v-\psi_\disc)}{L^q(\O)}
\lesssim h^{1-\frac{n}{q}} q \norm{F}{}.
\end{align*}
Minimising $q\mapsto h^{1-\frac{n}{q}}q$ over $q\in [1,\infty)$ if $n=2$, or taking $q=6$ if $n=3$,
yields
\begin{align*}
\norm{\Pi_\disc(v-\psi_\disc)}{L^\infty(\O)}
\lesssim \norm{F}{}\left\{\begin{array}{ll}h|\ln(h)|&\mbox{ if $n=2$},\\
h^{1/2}&\mbox{ if $n=3$}.\end{array}\right.
\end{align*}
This concludes \eqref{Linfty.error} since $\Pi_\disc v=\psi_\mesh$.
Estimate \eqref{Linfty.est.HMM} follows from \eqref{Linfty.error} by using the triangle inequality,
the estimate $\norm{\psi_\mesh}{L^\infty(\O)}\le \norm{\psi}{L^\infty(\O)}\lesssim \norm{F}{}$ and
the property $\max(h^{1/2},h|\ln(h)|)\lesssim 1$.
\end{proof}

\thanks{\textbf{Acknowledgement}: The first author acknowledges the funding support from the Australian Government through the Australian Research Council's Discovery Projects funding scheme (pro\-ject number DP170100605). 
The second and third authors acknowledge the funding support from the DST project SR/S4/MS/808/12.}

\bibliographystyle{abbrv}
\bibliography{jerome_nn_control}

\end{document}

%% file: fig-triangles_benchmark.pdf_t
\begin{picture}(0,0)%
\includegraphics{fig-triangles_benchmark.pdf}%
\end{picture}%
\setlength{\unitlength}{4144sp}%
\begingroup\makeatletter\ifx\SetFigFont\undefined%
\gdef\SetFigFont#1#2#3#4#5{%
  \reset@font\fontsize{#1}{#2pt}%
  \fontfamily{#3}\fontseries{#4}\fontshape{#5}%
  \selectfont}%
\fi\endgroup%
\begin{picture}(4921,4922)(-11,-7173)
\end{picture}%

%% file: fig-cart_shifted.pdf_t
\begin{picture}(0,0)%
\includegraphics{fig-cart_shifted.pdf}%
\end{picture}%
\setlength{\unitlength}{4144sp}%
\begingroup\makeatletter\ifx\SetFigFont\undefined%
\gdef\SetFigFont#1#2#3#4#5{%
  \reset@font\fontsize{#1}{#2pt}%
  \fontfamily{#3}\fontseries{#4}\fontshape{#5}%
  \selectfont}%
\fi\endgroup%
\begin{picture}(4309,4310)(-11,-7173)
\end{picture}%